\documentclass[envcountsame,envcountsect]{svproc}

\usepackage[english]{babel} 

\usepackage{amssymb,amsmath}

\usepackage{bbm} 
\usepackage{enumerate}
\usepackage{graphicx}
\usepackage{float}
\usepackage{enumitem} 
\usepackage{mathtools} 
\usepackage{xcolor} 
\usepackage{soul} 
\sethlcolor{lightgray} 
\usepackage[linkcolor=blue]{hyperref}
\hypersetup{
  colorlinks   = true, 
  urlcolor     = blue, 
  linkcolor    = blue, 
  citecolor   = red 
}
\usepackage{mathrsfs} 
\usepackage{cite} 

 \usepackage{multicol}
 \usepackage[bottom]{footmisc}

\usepackage{url}

\usepackage{tikz-cd}
\usepackage{amsmath}
\usepackage{tikz}
\usepackage{pgfplots}
\usetikzlibrary{positioning}
\tikzstyle{line} = [draw, -latex']
\pgfplotsset{compat=newest}
\pgfplotsset{plot coordinates/math parser=false}
\newlength\figureheight
\newlength\figurewidth

\newcommand{\N}{\mathbb{N}}
\newcommand{\Z}{\mathbb{Z}}
\newcommand{\R}{\mathbb{R}}
\newcommand{\C}{\mathbb{C}}

\newcommand{\udef}{\mathllap} 

\newcommand{\vv}[2]{v_{#1}^{(#2)}}

\newcommand{\T}{\mathbb{T}}
\newcommand{\wt}{\widetilde}
\newcommand{\w}{\widetilde{w}}

\newcommand\numberthis{\addtocounter{equation}{1}\tag{\theequation}} 

\newcommand{\tw}{\textcolor{black}}
\newcommand{\mumin}{\mu_{\min}}

\makeatletter
\def\blfootnote{\gdef\@thefnmark{}\@footnotetext}
\makeatother

\DeclareMathOperator{\rank}{{\mathrm{rank}}}
\DeclareMathOperator{\linspan}{{\mathrm{span}}}

\DeclareMathOperator{\re}{{\mathrm{Re}}}

\DeclareMathOperator{\diag}{{\mathrm{diag}}}

\numberwithin{equation}{section} 

\begin{document}
\mainmatter              
\title{Sharp Decay Estimates in Local Sensitivity Analysis for Evolution Equations with Uncertainties: from ODEs to Linear Kinetic Equations}
\titlerunning{Sharp Decay Estimates: from ODEs to Linear Kinetic Equations}  
%
\author{
Anton Arnold\inst{1} \and Shi Jin\inst{2}
\and Tobias W\"{o}hrer\inst{3}}
\authorrunning{A. Arnold, S. Jin, T. W\"{o}hrer} 
%
%
\institute{Institute for Analysis and Scientific Computing, TU Vienna,\\ Wiedner Hauptstra{\ss}e 8--10, 1040 Vienna, Austria,\\
\email{anton.arnold@tuwien.ac.at},\\ website:
\url{https://www.asc.tuwien.ac.at/~arnold/}
\and
School of Mathematical Sciences, Institute of Natural Sciences, MOE-LSC, Shanghai Jiao Tong University,\\ Shanghai 200240, China,\\
\email{shijin-m@sjtu.edu.cn}\\
\and Institute for Analysis and Scientific Computing, TU Vienna,\\ Wiedner Hauptstra{\ss}e 8--10, 1040 Vienna, Austria,\\
\email{tobias.woehrer@tuwien.ac.at}
}

\maketitle
\begin{abstract}
We review the Lyapunov functional method for linear ODEs and give an explicit construction of such functionals that yields sharp decay estimates, including an extension to defective ODE systems. As an application, we consider three evolution equations, namely the linear convection-diffusion equation, the two velocity BGK model and the Fokker--Planck equation.

Adding an uncertainty parameter to the equations and analyzing its linear sensitivity leads to defective ODE systems. By applying the Lyapunov functional construction, we prove sharp long time behavior of order $(1+t^{M})e^{-\mu t}$, where $M$ is the defect and $\mu$ is the spectral gap of the system. The appearance of the uncertainty parameter in the three applications makes it important to have decay estimates that are uniform in the \emph{non-defective limit}.

\blfootnote{The first author was partially supported by the FWF (Austrian Science Fund) funded SFB \#F65. The first and the third authors were partially supported by the FWF-doctoral school W 1245 as well as the University of Madison-Wisconsin, where part of this work was carried out. The second author  was partially supported by NSFC grants Nos. 31571071 and 11871297.}
\end{abstract}
\tableofcontents
\section{Introduction}
Kinetic models arise from \tw{mesoscopic} approximations of particle systems,
as such, they are not first principle equations, thus contain empirical
coefficients such as collision kernels in the Boltzmann equation, scattering
coefficients in transport equations, forcing or source terms, and measurement
errors in initial and boundary data, etc.  Such errors can be modeled by
uncertainties, or random inputs.  Quantifying these uncertainties have
important industrial and practical applications, in order to identify the
sensitivities of input parameters, validate the models, conduct risk management, and ultimately improve these models.

In recent years one has seen activities in conducting uncertainty quantifications (UQ) for kinetic equations, see \cite{Hu2017} for a recent review. One of the important analysis in UQ is the so-called local sensitivity analysis, in which one aims to understand how sensitive the solution depends on the input parameters \cite{smithbook}. For kinetic equations, a major tool to conduct sensitivity analysis for
random kinetic equations has been the coercivity, or more generally, hypocoercivity, which originated in the study of long-time behavior of kinetic equations (see
\cite{Villani-book, Dolbeault2015, MB, MN}). In such analysis, by using the
hypocoercivity of the kinetic operator, in a perturbative setting, namely,
considering solutions near the global equilibrium (see \cite{Guo-NS}), one can establish the long time convergence toward the local equilibrium with an exponential
time decay rate. Such analysis has been extended to kinetic equations with
random inputs, in both linear (see \cite{JLM, LiWang}) and nonlinear (see \cite{JinZhu, LiuJin, ShuJin}) settings. For stochastic Galerkin methods, hypocoercivity analysis even leads to exponential decay of numerical errors \cite{LiuJin, ShuJin}, while in classical numerical analysis one often obtains errors that grow exponentially in time. In these works, however,
the decay rates were not sharp.

Over the last two decades, entropy methods have become an important and robust tool to prove exponential convergence to equilibrium in kinetic and parabolic equations (see \cite{To,AMTU01,DeVi05,JuengelBook,Dolbeault2015}). But sharpness of the decay rate is only known in few cases (see \cite{AMTU01} for the situation in Fokker--Planck equations). For linear finite dimensional ODEs, however, a method of constructing Lyapunov functionals to reveal optimal decay rates has been known for a long time, see \cite{Arnold1978}, \S 22.4.

More recently in \cite{Arnold2014}, such strategies were transferred to Fokker-Planck (FP) equations on $\R^d$ and used to estimate the decay behavior of their solutions. For both, the ODE and the FP setting, one obtains the sharp \emph{exponential} decay rate, as long as none of the eigenvalues determining the spectral gap of the generator is defective\footnote{An eigenvalue is \emph{defective}, if its algebraic multiplicity is strictly greater than its geometric multiplicity.}. Recently this method was applied to PDEs that allow for a modal decomposition, like kinetic BGK models on the torus (see \cite{Achleitner2016,Achleitner2018}). They are relaxation-type models for collisional gases, introduced by the physicists Bhatnagar, Gross and Krook in \cite{BGK}.

In the defective case, however, the sharp decay behavior is of the form of a \emph{polynomial times an exponential}, and different strategies have to be applied.

In order to catch the sharp decay behavior in the case of a defective FP equation, one can use the spectral properties of the FP operator to split the solution into two subspace-invariant parts: The first one corresponds to the spectral gap and is finite dimensional; there the sharp (defective) decay behavior can be computed explicitly. The second part of the solutions corresponds to a subspace ``away'' from the spectral gap, and it has a faster exponential decay. This approach gives sharp decay functions for defective FP equations for various entropies as shown in \cite{AEW17}.

Alternatively, one can extend the Lyapunov functional by allowing it to be time dependent, see \cite{Monmarche2015}. In \S \ref{sec:ODE} of this paper we will translate that strategy from linear Fokker--Planck equations in $\R^d$ (as in Corollary 12 of \cite{Monmarche2015}) to the ODE setting. We shall also refine the method such that it can yield uniform decay bounds in the non-defective limit. This extension will be crucial for our PDE-applications presented in \S 3--\S 5: a convection-diffusion equation, a BGK model and a linear Fokker--Planck equation on $\R$ respectively. There we shall allow for uncertainty in the model coefficients and carry out a first (and for the convection-diffusion equations also second) order sensitivity analysis. In a local sensitivity analysis one tries to estimate the behavior of the (higher order)  derivatives of the solution
with respect to the input variables \cite{smithbook}. Estimates of such derivatives are not only important to
assess the sensitivity of the solution on the input parameters, they also provide regularity of the solution in the parameter space which is important to determine the convergence order of numerical
approximations in the random space \cite{gunzburger2014stochastic, JLM, jin2017asymptotic}.  In the Fourier space, the resulting evolution equations for the parametric derivatives are  mostly defective systems for which the sharp decay estimates can be obtained by using the Lyapunov functional approaches for defective deterministic systems. We would also like to point out that for the case of linear Fokker-Planck equation with a random drift, studied in \S 5, the global equilibrium is also random, while in previous sensitivity analysis  for uncertain kinetic equations the global equilibria were all deterministic \cite{JinZhu,LiuJin}.

\section{Lyapunov functionals for defective ODEs}\label{sec:ODE}
In this section we first review (from \cite{Achleitner2016}) the Lyapunov functional method for non-defective ODEs and then extend it to the defective case. \tw{This is based on constructing a norm \emph{adapted to the problem}} that allows to recover the sharp decay behavior.\\

\subsection{Construction of Lyapunov fuctionals}\label{subsec:pmatrix}


Let the matrix $C\in\C^{d\times d}$ be positive stable, i.e.\ its eigenvalues satisfy $\re (\lambda_i)>0$ for $i=1,\ldots,d$, and let $\mu:=\min_{i=1,\ldots, d} \re (\lambda_i)>0$. We want to find a Lyapunov functional for the equation
\begin{align}\label{eq:ODE}
\frac{d}{dt}x(t) = -C x(t),\quad x\in\C^d, t\geq 0
\end{align}
that allows to deduce the sharp decay rate of solutions with energy-type estimates.
For the construction of this functionals we consider the Jordan transformation of the matrix $C^H$\tw{, denoting the Hermitian transpose of the matrix $C$ (with eigenvalues $\overline{\lambda_i}$).} We shall distinguish different cases of eigenvalue defectiveness:
\begin{align}\label{eq:V}
C^H = V \operatorname{diag}(J_1,\ldots,J_N) V^{-1},
\end{align} where $J_n$ for $n\in\{1,\ldots, N\}$ are the Jordan blocks of $C^H$ with length $l_n\in\{1,\ldots,d\}$. A Jordan block of length one is an eigenvalue as a diagonal element, and a  Jordan block $J_n$ of length $l_n>1$ corresponds to a chain of generalized eigenvectors of $C^H$ of order $k$ satisfying
\begin{align}\label{eq:geneigenvector}
C^H \vv{n}{k} = \overline{\lambda}_n \vv{n}{k} + \vv{n}{k-1}, \quad k\in\{1,\ldots,l_n-1\},
\end{align}
where $\vv{n}{0}$ is an eigenvector of $C^H$, corresponding to $\overline{\lambda}_n$.
We denote the (semi-)\\norm $$|x|^2_P:=x^H P x,$$ for a Hermitian positive (semi-)definite  matrix $P\in\C^{d\times d}$ to be defined.
\subsection*{Case 1: $J_n$ is a Jordan block of length $l_n=1$ with $\re(\lambda_n) \geq \mu$.}
We define the rank 1 matrix\footnote{For $v,w\in\C^d$ we denote $v\otimes w:= v\cdot w^H$ where $\cdot$ is the matrix-matrix multiplication.} $P_n:= v_n^{(0)}\otimes \vv{n}{0}$ and get (cf. (2.51) in \cite{Arnold2014})
\begin{align}\label{eq:case1}
\frac{d}{dt}|x(t)|^2_{P_n} = -x^H (C^H P_n + P_n C) x \leq -2\mu x^H P_n x = -2\mu |x(t)|_{P_n}^2.
\end{align}
\subsection*{Case 2: $J_n$ is a Jordan block of length $l_n>1$ with $\re(\lambda_n)>\mu$.} As in the proof of Lemma 4.3 in \cite{Arnold2014}, we choose the coefficients $b_n^i>0$ as
\begin{align*}
b_n^1:=1;&&b^j_n:=c_j(\tau_n)^{2(1-j)},\quad j\in\{2,\ldots,l_n\},
\end{align*}
where $c_1:=1$, $c_j:=1+(c_{j-1})^2$ for $j\in\{2,\ldots,l_n\}$ and $\tau_n:=2(\re(\lambda_n)-\mu)> 0$. Then the matrix
\begin{align*}
P_n:=\sum_{i=1}^{l_n} b_n^i \vv{n}{i-1}\otimes \vv{n}{i-1}
\end{align*}
satisfies 
\begin{align}\label{eq:Pnmatrixineq}
C^H P_n+P_n C \geq 2\mu P_n
\end{align}
 and, as in \eqref{eq:case1}, one gets
\begin{align}\label{eq:Pndecay}
\frac{d}{dt}|x(t)|^2_{P_n} \leq -2\mu |x(t)|_{P_n}^2.
\end{align}

\subsection*{Case 3: $J_n$ is a Jordan block of length $l_n>1$ with $\re(\lambda_n)=\mu$.}
A translation of the strategy of Corollary 12 in \cite{Monmarche2015} to the ODE setting leads to the following construction. For each $m\in\{1,\ldots,l_n \}$, define the vector function
\begin{align}\label{eq:wnm}
w^m_n(t):= \sum^m_{k=1} \frac{t^{m-k}}{(m-k)!} \vv{n}{k-1}, \quad t\geq 0.
\end{align}
For $m\in \{2,\ldots,l_n\}$ we have
\begin{align*}
\frac{d}{dt}w^m_n(t) = \sum^{m-1}_{k=1} \frac{t^{m-k-1}}{(m-k-1)!}v_n^{(k-1)},
\end{align*}
from which it follows (using \eqref{eq:geneigenvector}) that
\begin{align}\label{eq:connectionCw}
C^H w^m_n(t) - \overline{\lambda_n} w^m_n(t) = \sum_{k=2}^m \frac{t^{m-k}}{(m-k)!}\vv{n}{k-2}= \frac{d}{dt}w^m_n(t).
\end{align}
Next we define the time dependent Hermitian positive semi-definite matrix
\begin{align}\label{eq:defPmn}
P^m_n(t):= w^m_n(t)\otimes w^m_n(t),\quad m\in\{1,\ldots,l_n\}.
\end{align}
Notice that $w^1_n(t) = \vv{n}{0}$. So for $m=1$ the computation for the estimate of $\frac{d}{dt}|x(t)|^2_{P^1_n(t)}$ is the same as in \eqref{eq:case1} (with equality since $\re (\lambda_n) = \mu$). For $m\in\{2,\ldots,l_n\}$ we use the identity \eqref{eq:connectionCw}, and compute
\begin{align*}
\frac{d}{dt} |x(t)|^2_{P^m_n(t)} &= \dot{x}^H(t) P^m_n(t) x(t) + x^H(t) P^m_n(t) \dot{x}(t) + x^H(t) \dot{P^m_n}(t)x(t)\\
&= - x^H(t) [ C^H w^m_n(t)\otimes w^m_n(t) + w^m_n(t)\otimes w^m_n(t) C ]x(t)\\
&\qquad + x^H(t) \left[(C^H w^m_n(t) - \overline{\lambda_n} w^m_n(t) ) \otimes w^m_n(t) \right]x(t)\\
&\qquad+ x^H(t)\left[ w^m_n(t) \otimes (C^H w^m_n(t) - \overline{\lambda_n} w^m_n(t)) \right]x(t)\\
&=  -2\mu x^H(t) w_n^m(t)\otimes w_n^m(t) x(t) = -2\mu |x(t)|^2_{P_n^m(t)}
\end{align*}
and directly obtain
\begin{align}\label{eq:case3}
|x(t)|_{P^m_n(t)}^2 = e^{-2\mu t} |x(0)|^2_{P^m_n(0)},\quad t\geq 0.
\end{align}
For arbitrary $\beta_n^m>0$, define
\begin{align}\label{eq:Pndefective}
P_n(t) := \sum^{l_n}_{m=1} \beta_n^m P^m_n(t).
\end{align}
We have $\linspan\{\vv{n}{0},\ldots,\vv{n}{l_n-1}\} = \linspan\{w^1_n(t),\ldots, w^{l_n}_n(t)\}$ for all $t\geq 0$, since the transformation matrix between these two sets is given by $e^{-\overline{\lambda_n} t} e^{J_n t}$.
Hence, the matrix $P_n(t)$ is positive definite on the subspace $\linspan\{\vv{n}{0},\ldots,\vv{n}{l_n-1}\}$. In the corresponding semi-norm $|\cdot|_{P_n(t)}^2$ the solution $x(t)$ satisfies
\begin{align*}
|x(t)|^2_{P_n(t)} = e^{-2 \mu t} |x(0)|^2_{P_n(0)}, \quad t\geq 0.
\end{align*}
For later convenience, we denote
\begin{align}\label{eq:defIndex}
I_\mu:=\{n\in\{1,\ldots,N\} \mid l_n>1, \re(\lambda_n) = \mu\},
\end{align}
to collect all indices with non-trivial Jordan blocks corresponding to $\mu$, i.e. corresponding to the above Case 3.
\subsection*{Combining the three cases:}
Now let us define
\begin{align}\label{eq:defP}
P(t) &:= \sum_{n\not\in I_\mu} \beta_n P_n  +\sum_{n\in I_\mu}P_n(t) =
\sum_{n\not\in I_\mu} \beta_n P_n  +\sum_{n\in I_\mu}\sum_{m=1}^{l_n} \beta_n^m P_n^m(t),
\end{align}
where $P_n$ and $P_n^m(t)$ are chosen, depending on the above Cases 1--3 of the corresponding Jordan block $J_n$. The weights $\beta_n>0$ are arbitrary in Cases 1--2, and the (arbitrary) $\beta_n^m>0$ pertain to Case 3. The matrix $P(t)$ is positive definite for every $t\geq 0$, since it has full rank by construction and it is the sum of positive semi-definite matrices. It satisfies
\begin{align}\label{eq:dtP}
\frac{d}{dt} |x(t)|^2_{P(t)} \leq -2\mu |x(t)|^2_{P(t)},
\end{align} and by applying Gronwall's lemma, we conclude:
\begin{lemma} \label{lem:Pexists}
Let $C\in\C^{d\times d}$ be positive stable and let $\mu>0$ be the smallest real part of all eigenvalues. Let $\diag(J_1,\ldots,J_N)$ be the Jordan normal form of $C^H$, where $J_n$ for $n\in\{1,\ldots,N\}$ is a Jordan block of length $l_n$ with eigenvalue $\overline{\lambda_n}$.
\begin{enumerate}
\item If all eigenvalues with real part equal to $\mu$ are non-defective, i.e.\ $I_\mu=\emptyset$, then there exists a time-independent Hermitian positive definite matrix $P\in\C^{d\times d}$, such that the solutions to \eqref{eq:ODE} satisfy
\begin{align}\label{eq:Pconstantdecay}
|x(t)|^2_P \leq e^{-2\mu t} |x(0)|^2_P.
\end{align}
\item If at least one eigenvalue with real part equal to $\mu$ is defective, i.e.\ $I_\mu \neq \emptyset$, then there exists a time-dependent matrix $P(t)\in\C^{d\times d}$, which is Hermitian positive definite for all $t\geq 0$ such that the solutions to \eqref{eq:ODE} satisfy
\begin{align}\label{eq:decayofPt}
|x(t)|^2_{P(t)} \leq e^{-2\mu t} |x(0)|^2_{P(0)}.
\end{align}
\end{enumerate}
\end{lemma}
For further details on the algebraic interpretation of the time-independent matrix $P$, we refer to the remarks following Lemma 4.3 in \cite{Arnold2014}. See Example 2.2 in \cite{AAS} (with $\omega\neq 0$) for an ODE example and the relevance of the modified (time-independent) $P$-norm for the trajectories of the ODE.

\begin{remark}
The matrix $P(t)$ is --- with the construction described above --- not unique. For one, arbitrary coefficients $\beta_n,\beta_n^m>0$ in the definition of $P(t)$ in \eqref{eq:defP} are admissible. Secondly, the construction depends on the specific choice of (generalized) eigenvectors fixed in \eqref{eq:geneigenvector}.
\end{remark}

\begin{remark}\label{rem:Pmatrixform}
The matrix $P(t)$ can also be written as a matrix product:
\begin{align*}
P(t) = V e^{Jt}\Sigma(t)B(Ve^{Jt})^H,
\end{align*}
with $V$ and $J$ from \eqref{eq:V},
\begin{align*}
B:=\diag(\underbrace{\beta_1^1,\ldots,\beta_1^{l_1}}_{l_1 \text{ entries}},\ldots,\underbrace{\beta_N^1,\ldots,\beta_N^{l_N}}_{l_N \text{ entries}})\in \R^{d\times d},
\end{align*}
with notation $\beta_n^m:=\beta_n$ for each $n\not\in I_\mu$ and corresponding $m\in\{1,\ldots,l_n\}$ and
\begin{align*}
\Sigma(t): = \diag(\underbrace{e^{-2\re(\lambda_1)t},\ldots,e^{-2\re(\lambda_1) t}}_{l_1 \text{ times}},\ldots,\underbrace{e^{-2\re(\lambda_N) t},\ldots,e^{-2\re(\lambda_N) t}}_{l_N \text{ times}})\in \C^{d\times d}.
\end{align*}
This representation of $P(t)$ directly implies that $\det P(t) \equiv \det P(0)$ (cf.\ Fig.\ \ref{fig:Pplottime}).
\end{remark}

In the following remark and in Example \ref{ex:geo}, we investigate the geometry of the modified norms of Lemma \ref{lem:Pexists}.
\begin{remark}\label{rem:impactangle}
For an ODE \eqref{eq:ODE}, we distinguish different eigenvalue settings of the matrix $C\in\C^{d\times d}$, in order to isolate the interesting phenomena.

\begin{itemize}
\item \emph{Case 1: $C$ is in Case 1 of Lemma \ref{lem:Pexists}:} Due to Lemma \ref{lem:Pexists}, there exists a time-independent $P$-norm such that solutions decay as \eqref{eq:Pconstantdecay}. The geometric reason for the strict decay is the following: This specific norm is modified such that the trajectories of solutions to the ODE are never tangential to the $P$-norm level curves $\{x\in\C^d \mid |x|_P^2 = \emph{const.}\}$ (cf. Fig.\ \ref{fig:Pploteps}).

To prove this, denote $f(x):=x^HPx$. Then, the normal vector of the $P$-norm level curve at point $x\in\C^d \setminus \{0\}$ is given as the $P$-norm gradient of $f(x)$, i.e.\ $\eta(x):=\nabla_P f(x) = P^{-1}\nabla f(x) = 2x$ (see, e.g.\ (2.1.13) in \cite{jost} for gradients of Riemannian manifolds). The (backwards-in-time facing) solution tangent vector at point $x$ is given as $-\dot{x}=Cx$.  Due to the matrix inequality \eqref{eq:Pnmatrixineq},  $\eta(x)$ and $-\dot{x}$ are never perpendicular, i.e. the $P$-norm angle between them is bounded from below:
\begin{align*}
\frac{\langle x, Cx\rangle_P }{|x|_P |Cx|_P} &= \frac{x^HPC x }{|x|_P |Cx|_P} =  \frac12\frac{x^H(C^TP + PC)x }{|x|_P |Cx|_P} \\
&\geq \mu \frac{x^HPx }{|x|_P |Cx|_P} = \mu \frac{|x|_P }{ |Cx|_P} \geq \frac{\mu}{ |C|_P}  >0,
\end{align*}
where $|C|_P$ denotes the matrix norm induced by the vector norm $|x|_P$.

\item \emph{Case 2: $C$ is in Case 2 of Lemma \ref{lem:Pexists}}: First, for the time-independent matrix $P_\epsilon$, as defined in \cite{Arnold2014}, Lemma 4.3, the analogous result as for the above case is true. The calculation is identical, up to replacing $\mu$ by $\mu-\epsilon$ (cf. Example \ref{ex:geo}). Thus, also in the defective case, the \emph{solutions are never tangential to the level curves of the $P_\epsilon$-norm. }

The $P(t)$-norm of Case 2 in Lemma \ref{lem:Pexists} has a different geometric effect on solutions due to its time-dependency. In fact, a solution $x(t)$ can even be tangential to the $P(t)$-norm level curves for all times, while still maintaining sharp exponential decay, as Example \ref{ex:geo} below shows.

\item \emph{Case 3: All eigenvalues of $C$ have real part $\mu$ (defective or non-defective)}: For $t\geq 0$, \emph{the $P(t)$-angle between $\eta(x(t))$ and $-\dot{x}(t)$ stays constant}, i.e.\
\begin{align}
\frac{\langle x(t),Cx(t)\rangle_{P(t)}}{|x(t)|_{P(t)}|Cx(t)|_{P(t)}} = \frac{\langle x(0),Cx(0)\rangle_{P(0)}}{|x(0)|_{P(0)}|Cx(0)|_{P(0)}}.
\end{align}
Indeed, in this case, \eqref{eq:dtP} becomes an equality (with $P(t)\equiv P(0)$, if all eigenvalues are non-defective), and hence $|x(t)|_{P(t)}^2 = e^{-2\mu t}|x(0)|_{P(0)}^2$ for all $t\geq 0$. \tw{In the following computation we use the polarization identity in the first and last step. The second identity follows from the fact that, if $x(t)$ is a solution to \eqref{eq:ODE}, so is $(I+C)x(t)$ and $(I+iC)x(t)$ (with initial conditions $(I+C)x(0)$ and $(I+iC)x(0)$, respectively):}
\begin{align*}
&\frac{\langle x(t),Cx(t)\rangle_{P(t)}}{|x(t)|_{P(t)}|Cx(t)|_{P(t)}} = \frac14\frac{ |x(t) + Cx(t)|_{P(t)}^2 - |x(t) - Cx(t)|_{P(t)}^2 }{e^{-\mu t}|x(0)|_{P(0)} e^{-\mu t}|Cx(0)|_{P(0)}}\\
&\qquad \qquad+\frac14 \frac{i|x(t)-iCx(t)|_{P(t)}^2 - i|x(t) + iCx(t)|_{P(t)}^2}{e^{-\mu t}|x(0)|_{P(0)} e^{-\mu t}|Cx(0)|_{P(0)}}\\
&\qquad  =\frac14 \frac{ e^{-2\mu t}|x(0) + Cx(0)|_{P(0)}^2 -e^{-2\mu t} |x(0) - Cx(0)|_{P(0)}^2}{e^{-\mu t}|x(0)|_{P(0)} e^{-\mu t}|Cx(0)|_{P(0)}}\\
&\qquad \qquad+\frac14 \frac{ e^{-2\mu t}i|x(0)-iCx(0)|_{P(0)}^2 -  e^{-2\mu t}i|x(0) + iCx(0)|_{P(0)}^2}{e^{-\mu t}|x(0)|_{P(0)} e^{-\mu t}|Cx(0)|_{P(0)}}\\
&\qquad=\frac{\langle x(0),Cx(0)\rangle_{P(0)}}{|x(0)|_{P(0)}|Cx(0)|_{P(0)}}
\end{align*}
for all $t\geq 0$.
\end{itemize}
\end{remark}

\begin{example} \label{ex:geo}
	Consider the IVP $\dot{x}=-Cx$, $x(0)=(6, 6)^T$, with matrix
	\begin{align*}
	C=\begin{pmatrix}
	1& \frac12\\
	-\frac12& 0
	\end{pmatrix},
	\end{align*}
	which has the defective eigenvalue and spectral gap $\lambda=\mu=\frac12$ and the (generalized) eigenvectors
	\begin{align*}
	w^{(0)}=\frac{1}{\sqrt{2}}\begin{pmatrix}1,& -1 \end{pmatrix}^T,&& w^{(1)}=\frac{1}{\sqrt{2}}\begin{pmatrix}1,&1 \end{pmatrix}^T.
	\end{align*}
	Our goal is to obtain a better geometric understanding of the necessity of a time-dependent norm for sharp decay estimates of solutions. To this end, we compare the here presented Lyapunov functional constructions with functionals considered in \cite{Arnold2014}.
	
	The naive approach of using the Euclidean norm of the solution $x(t)$ exhibits non strict decay (the dashed curve in Fig.\ \ref{fig:Pploteuclid} has a horizontal tangent at $t=1$). The time-independent norm $|\cdot|_{P_\epsilon}$, as defined in \cite{Arnold2014}, Lemma 4.3, with the matrix
	\begin{align*} 
	P_\epsilon = 
	\frac{1}{\sqrt{2}}\begin{pmatrix}
	\frac{1}{2 \epsilon^2}+1 & \frac{1}{2 \epsilon^2}-1\\
	\frac{1}{2 \epsilon^2}-1& \frac{1}{2 \epsilon^2}+1
	\end{pmatrix},\quad \epsilon>0,
	\end{align*}
	yields uniform exponential decay, but the rate $\mu-\epsilon$ is not sharp. See Fig.\ \ref{fig:Pploteuclid} for the decay plot and Fig.\ \ref{fig:Pploteps} for a geometric reasoning why a modified norm can yield exponential decay.
	
	The $P(t)$-norm, with matrix
	\begin{align*}
	P(t)= 
	\frac12 \begin{pmatrix}
	t^2 +2t + 2& t^2\\
	t^2& t^2-2t+2
	\end{pmatrix},
	\end{align*}
	as defined in \eqref{eq:defP} (with weights $\beta^1=\beta^2=1$),
	provides the sharp exponential decay $$|x(t)|_{P(t)}^2 = e^{-t}|x(0)|_{P(0)}^2 = e^{-t}|x(0)|^2_2,\qquad t\geq 0.$$ 
	See Figures \ref{fig:Pplottime}--\ref{fig:Pplotexp} for the  geometric evolution of the $P(t)$-norm.
	
	Choosing the initial condition $\tilde{x}(0)= (0,7)^T$, yields the solution $\tilde{x}(t)$ that is tangential to the $P(t)$-norm level curve for each $t\geq 0$, while the exponential decay of the solution in $P(t)$-norm is still sharp, see Figure \ref{fig:Pplotexp_tang}.
	
	\hfill$\Diamond$
	\begin{figure}[ht]
		\centering
		\includegraphics[scale=0.6]{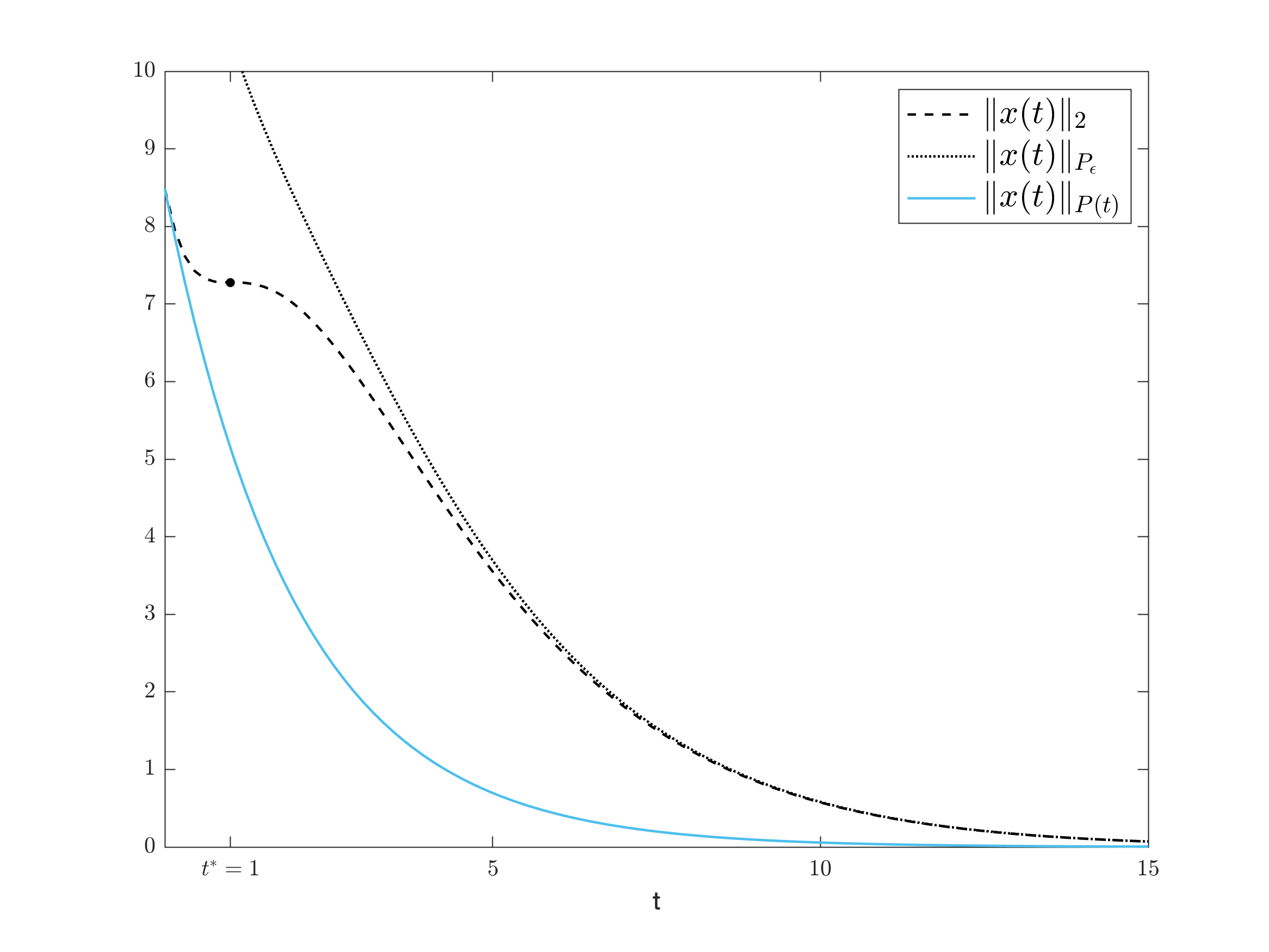}
		\caption{The dashed line shows the decay of the solution in the Euclidean norm. It initially exhibits a wavy behavior, where at time $t^*=1$, there is no strict decay at all. The dotted line describes the solution in $P_\epsilon$-norm with $\epsilon=0.4$. It yields uniform exponential decay, however, the decay rate is not sharp. 
			The solid line shows the decay of $|x(t)|_{P(t)}$ with sharp exponential rate $e^{-\frac12 t}$.
		}
		\label{fig:Pploteuclid}
	\end{figure}
	
	\begin{figure}[ht]
		\centering
		\includegraphics[scale=0.6]{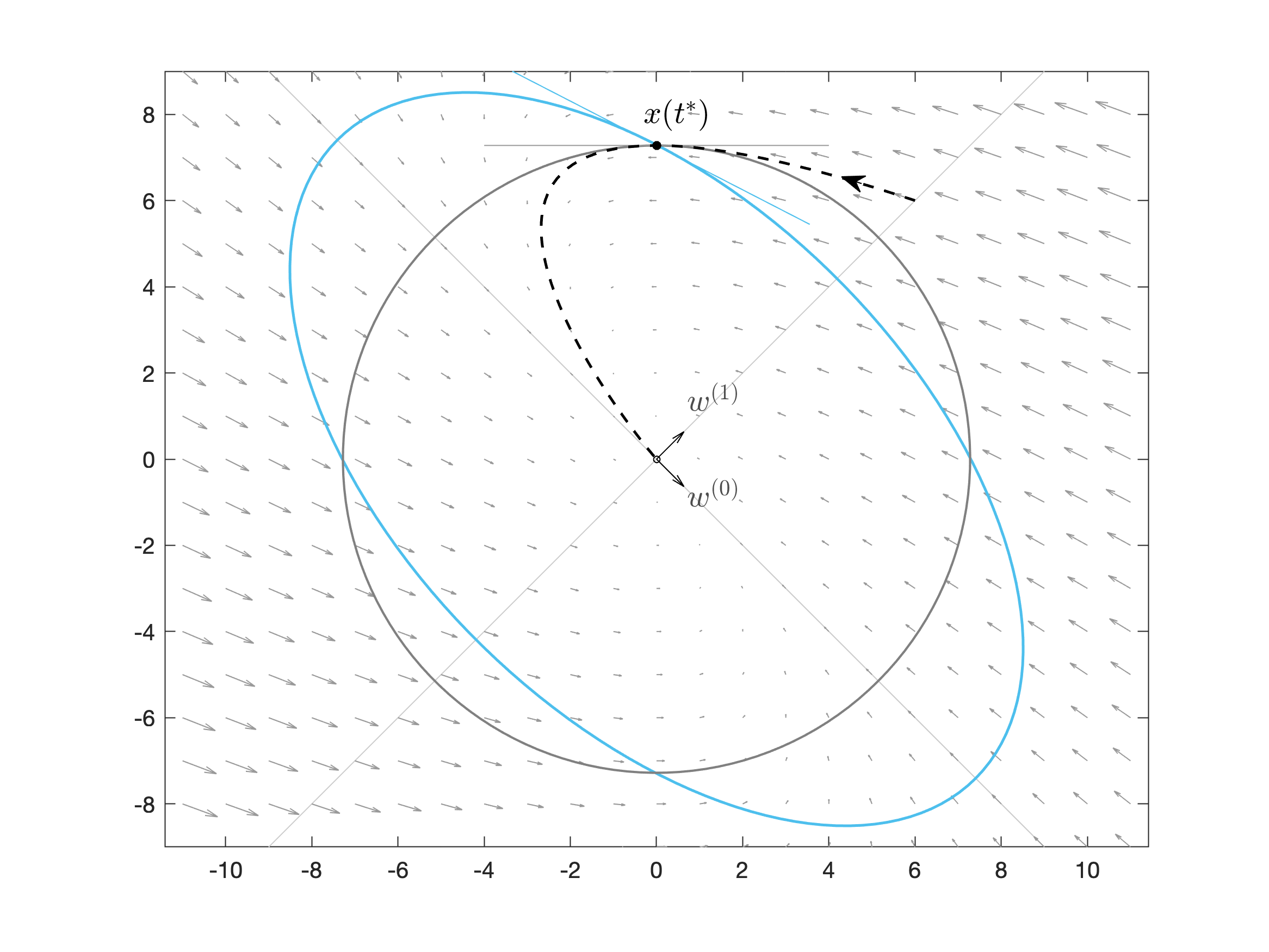}
		\caption{The dashed line shows the solution trajectory $x(t)$. At the marked point $x(t^*)$, the solution is tangential to the Euclidean level curve. This implies non-strict decay in the Euclidean norm (cp.\ Fig.\ \ref{fig:Pploteuclid}) at $t^*=1$. The ellipse represents a level curve of the $P_\epsilon$-norm (with $\epsilon=0.4$). It modifies the geometry such that the solution is never tangential to the level curves of $|\cdot|_{P_\epsilon}$. This assures strict exponential decay in the $P_\epsilon$-norm, however the rate is not sharp.}
		\label{fig:Pploteps}
	\end{figure}
	
	\begin{figure}[ht]
		\centering
		\includegraphics[scale=0.6]{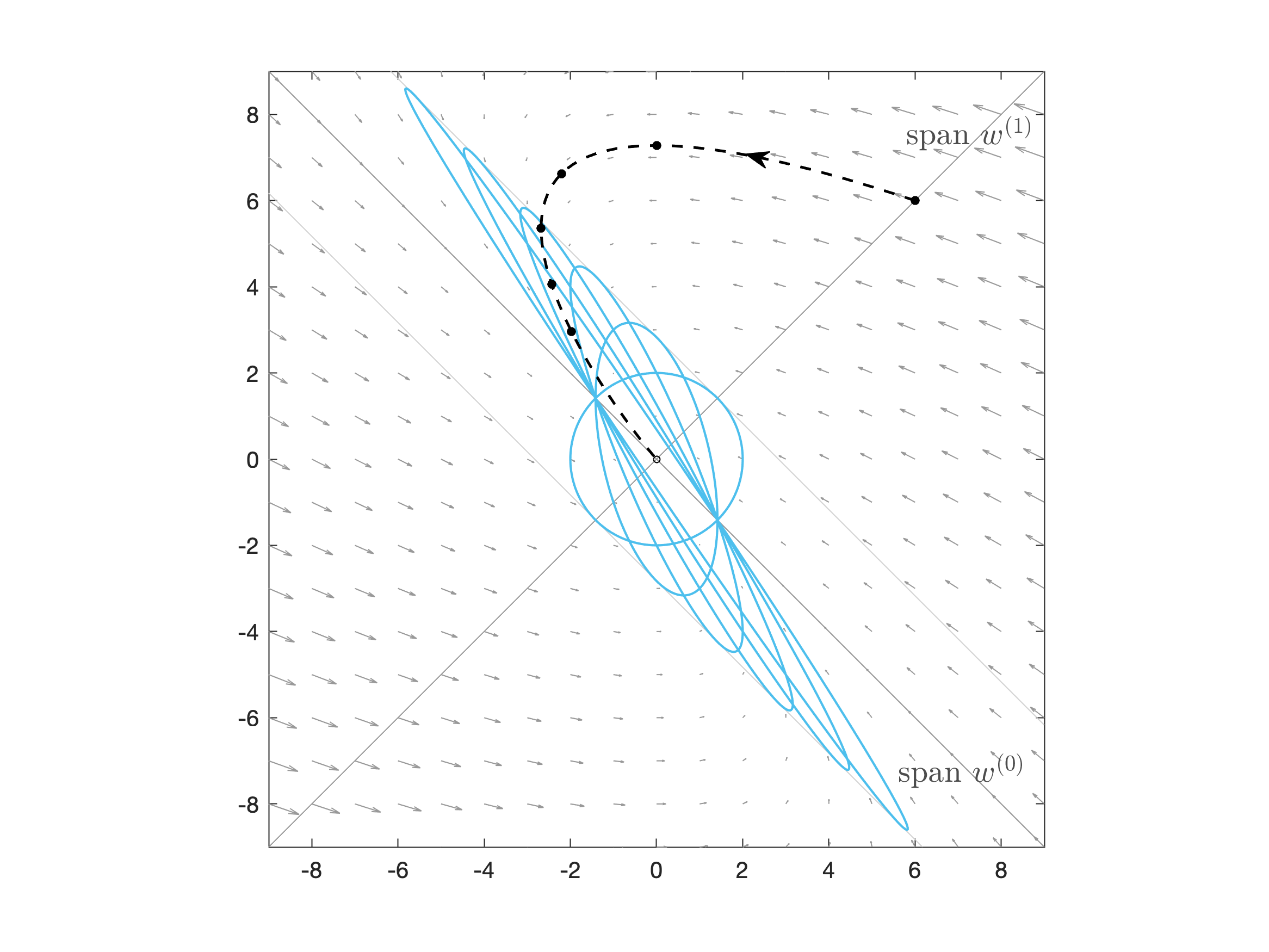}
		\caption{The dashed line in the plot describes the solution $x(t)$ with the marked points $x(0), x(1),\ldots,x(5)$. Additionally, the level curves of $\{x\in\R^2 \mid |x|_{P(t)}^2 = 4\}$ for $t=0,1,\ldots,5$ are plotted. In direction of the eigenvector of the matrix $C$, $w^{(0)}$, the distances stay constant in time. The area spanned by each ellipse-shaped level curve stays constant (cf. Remark \ref{rem:Pmatrixform}), while the semi-major axis of the ellipse stretches out and tilts towards the eigenvector axis $w^{(0)}$ as time increases. The stretch is linear in $t$ in the direction of $\pm w^{(0)}$: For arbitrary $t\geq 0$, the point $\pm\sqrt{2}(1-t, 1+t)^T$ is on the ellipse with the tangent $(\pm\sqrt{2}, \pm\sqrt{2})^T+\linspan\{w^{(0)}\}$.
		}
		\label{fig:Pplottime}
	\end{figure}
	
	\begin{figure}[ht]
		\centering
		\includegraphics[scale=0.6]{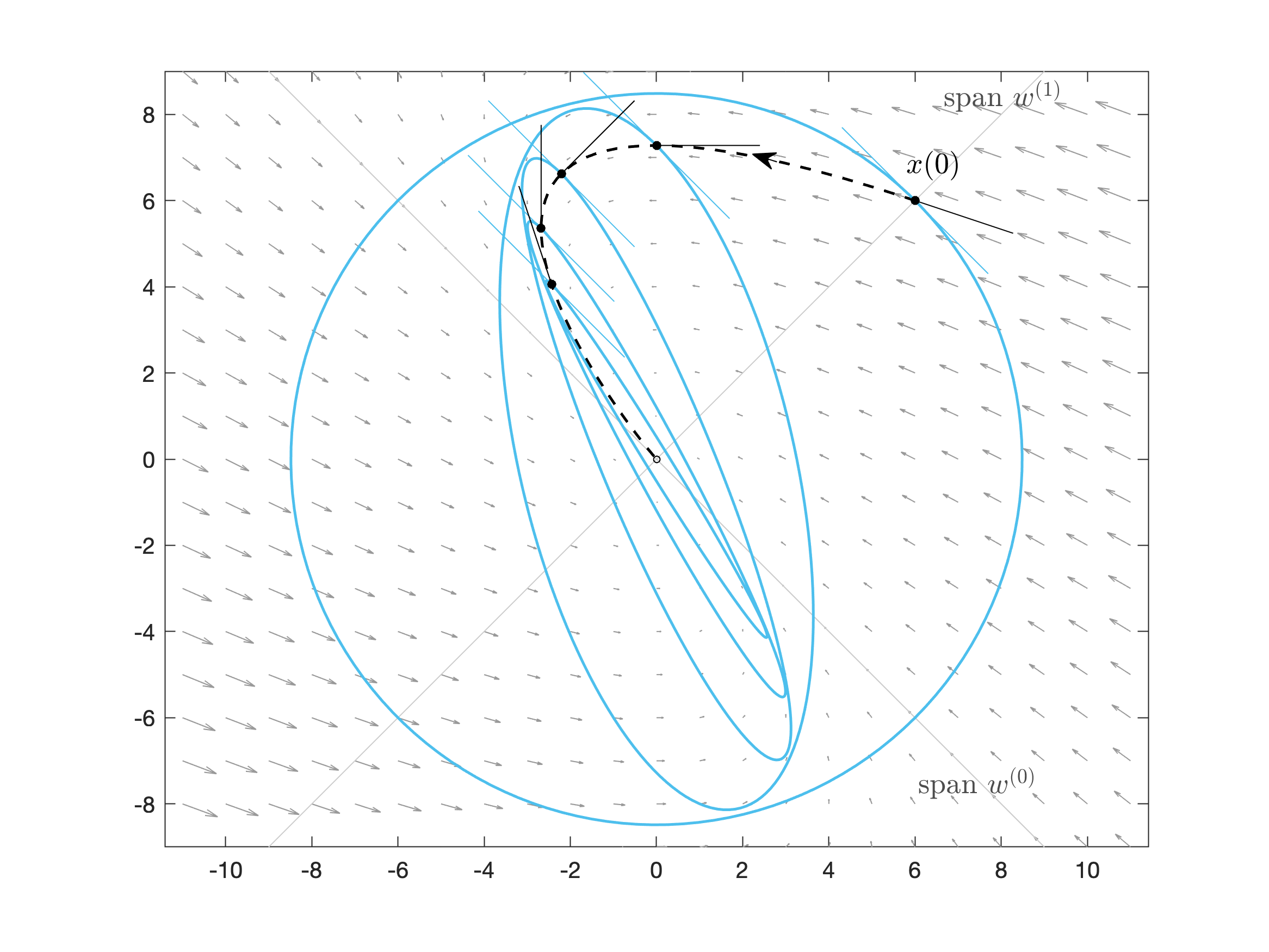}
		\caption{The level curves of $\{x\in\R^2 \mid |x|_{P(t)}^2 = e^{-t}|x(0)|^2_2\}$ for $t=0,1,\ldots,4$ are plotted (cp. Fig.\ \ref{fig:Pplottime}). They intersect with the solution trajectories exactly at the marked points $x(0),x(1),\ldots,x(4)$, which corresponds to the statement of Lemma \ref{lem:Pexists}, Case 2. Notice that the tangents of the level curves of $|\cdot|_{P(t)}$  at $x(t)$ are all parallel to each other. The intersection angle in the $P(t)$-norm is time-independent, see Remark \ref{rem:impactangle}, Case 3.
		}
		\label{fig:Pplotexp}
	\end{figure}
	
	\begin{figure}[ht]
		\centering
		\includegraphics[scale=0.6]{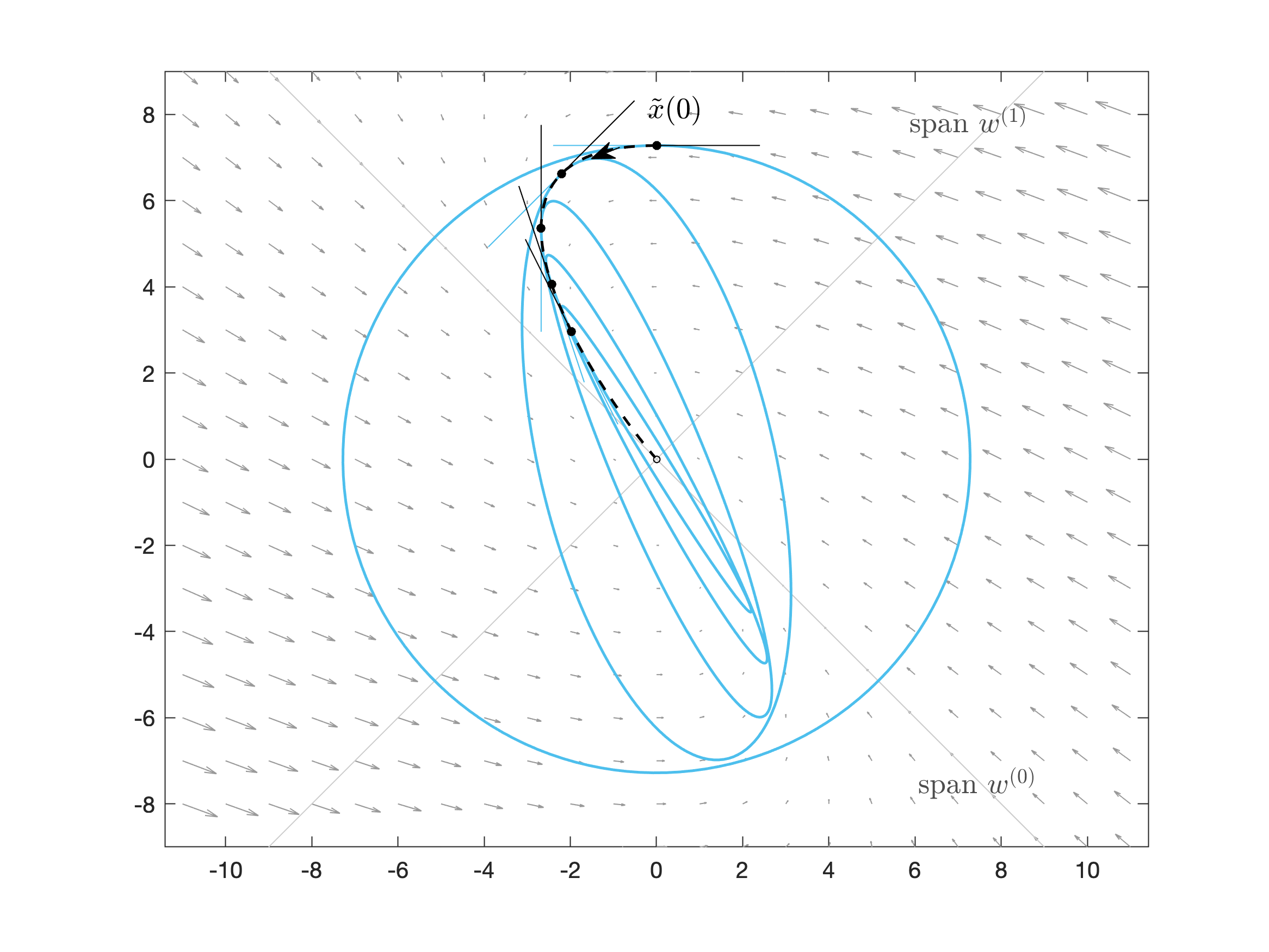} 
			\caption{The level curves of $\{x\in\R^2 \mid |x|_{P(t)}^2 = e^{-t}|\tilde{x}(0)|^2_2\}$ for $t=0,1,\ldots,4$ are plotted analogous to Fig.\ \ref{fig:Pplotexp}, but here for the initial value $\tilde{x}(0)=(0, 7)^T$. The solution $\tilde{x}(t)$ is tangential to the $P(t)$-norm level curves for each $t\geq 0$.
		}
		\label{fig:Pplotexp_tang}
	\end{figure}
	
	
\end{example}

Lemma \ref{lem:Pexists} shows that the $P(t)$-norm of any solution to \eqref{eq:ODE} decays exponentially. But due to the time dependence of the norm itself, it is not evident that this is an appropriate functional to capture the sharp decay rate of solutions. Hence, we shall next compare the $P(t)$-norm to the Euclidean norm.

An arbitrary Hermitian positive definite matrix $P\in\C^{n\times n}$ satisfies
\begin{align}\label{eq:normequiv}
\lambda^P_{\min} I \leq P \leq \lambda^P_{\max} I,
\end{align}
where $\lambda_{\min}^P$ is the smallest and $\lambda_{\max}^P$ is the largest eigenvalue of $P$. Using this inequality for $P(t)$, and the decay estimate \eqref{eq:decayofPt}, leads to the Euclidean decay estimate
\begin{align*}
|x(t)|^2_2 &\leq (\lambda_{\min}^{P(t)})^{-1}|x(t)|_{P(t)}^2 \leq (\lambda_{\min}^{P(t)})^{-1}e^{-2\mu t} |x(0)|_{P(0)}^2\\
& \leq (\lambda_{\min}^{P(t)})^{-1}\lambda^{P(0)}_{\max}e^{-2\mu t} |x(0)|_2^2.
\end{align*}
But here the decay behavior is ``hidden'' in the smallest eigenvalue of $P(t)$. The true qualitative behavior of $|x(t)|_2^2$ will be derived next.

The following technical Lemma \ref{lem:PtP0est} allows us, in the subsequent step, to estimate the time-dependency of the $P_n^m(t)$-semi-norm. The strategy of the proof has already been used for Corollary 12 of \cite{Monmarche2015} in the setting of Fokker--Planck equations on $\R^d$. However, our refined estimate here, for one, provides an upper bound for the time-dependence of a more general class of modified norms. And additionally, the constants appearing in the estimate are explicit, depending on $m$ and a parameter $\theta$ which, later on, allows to optimize the constants in the decay estimate of solutions to the ODE \eqref{eq:ODE}.
\begin{lemma}\label{lem:PtP0est}
For linearly independent vectors $v^1,\ldots,v^m\in\C^d$ define
	\begin{align}
	\hat{w}^m(t):&=\xi^m v^m + \sum_{j=1}^{m-1} \xi^j(t) v^j,
	\end{align}
	where $\xi^j(t)$ for $j\in\{1,\ldots,m-1\}$ are (arbitrary) real-valued polynomials in $t\geq 0$, and $\xi^m>0$. Furthermore let
	\begin{align*}
	\hat{P}^m(t):= 	\hat{w}^m(t)\otimes 	\hat{w}^m(t), \quad t\geq 0; \qquad Q^j:=v^j\otimes v^j, \quad j\in\{1,\ldots,m\}. 
	\end{align*}
	Then, the following inequality holds for every $x\in\C^d$, $\theta\in(0,1)$ and $t\geq 0:$
	\begin{align*}
	|x|_{\hat{P}^m(t)}^2 \geq (1-\theta) (\xi^m)^2|x|^2_{Q^m} - \left(\frac{(m-1)^2}{\theta}-1 \right)\sum_{k=1}^{m-1} (\xi^k(t))^2 |x|^2_{Q^k}.
	\end{align*}
\end{lemma}
Notice that the second ($t$-dependent) term of the r.h.s.\ involves only the semi-norms $|x|_{Q^1},\ldots,|x|_{Q^{m-1}}$.

The technical proof is deferred to Appendix \ref{pr:PtP0est}.

\begin{remark}\label{rem:PtP0est}
Lemma \ref{lem:PtP0est} is formulated in a general form, to be applicable also to \S \ref{sec:bgk} below. Here, we use it to estimate the time-dependency of the $P_n^m(t)$-semi-norms (as defined in \eqref{eq:defPmn}) for arbitrary $n\in I_\mu$ and corresponding $m\in\{2,\ldots,l_n\}$. In notation of Lemma \ref{lem:PtP0est}, choose $v^1=v_n^{(0)},\ldots,v^m=v_n^{(m-1)}$ and $\xi^k(t) = \frac{t^{m-k}}{(m-k)!}$ for $k\in\{1,\ldots,m\}$, which leads to $\hat{w}_n^m(t) = w_n^m(t)$. Then, \mbox{Lemma \ref{lem:PtP0est}}  yields
	\begin{align}\label{eq:PtP0est}
	\begin{aligned}	&\qquad |x|_{P_n^m(t)}^2\geq \\
		&\qquad \quad \left(1-\theta \right)|x|^2_{P_n^m(0)} - \left(\frac{(m-1)^2}{\theta}-1\right) \sum^{m-1}_{k=1} \left(\frac{t^{m-k}}{(m-k)!}\right)^2 |x|^2_{P_n^k(0)},
		\end{aligned}
		\end{align}
for every $n\in I_\mu$, all corresponding $m\in\{2,\ldots, l_n\}$, $x\in\C^d$, $\theta\in(0,1)$ and $t\geq 0$.
\end{remark}

In the next step, we combine the exponential decay in $P(t)$-norm of solutions, \eqref{eq:decayofPt} with the lower bound \eqref{eq:PtP0est}. This allows to estimate the $P(0)$-norm decay of solutions and, consequently, the decay behavior in the Euclidean norm. In contrast to the result stated in \cite{Monmarche2015} for the FP setting, we obtain an multiplicative constant $\mathscr{C}$ in the estimate depending explicitly on the maximal defect associated to $\mu$ and the choice of the weights $\beta_n^m$ of $P(t)$.

The freedom of choice in the weights and their influence on the constant $\mathscr{C}$ is of great importance in \S \ref{sec:hc}--\S \ref{sec:fpe}, as $\mathscr{C}$ will need to stay bounded in the \emph{non-defective limit} (see Example \ref{ex:Def1toDef0} below).
\begin{theorem}\label{th:EuclidDecay}
Let $C\in\C^{d\times d}$ be positive stable and let $\mu>0$ be the smallest real part of all eigenvalues. Let $M$ be the maximal size of a Jordan block associated to $\mu$ (i.e.\ the maximal defect associated to $\mu$ is $M-1$). Then there exists a constant $\mathscr{C}>0$, such that the solutions to \eqref{eq:ODE} satisfy
\begin{align}\label{eq:EuclidEst}
|x(t)|^2_2 \leq \mathscr{C}(1+t^{2(M-1)}) e^{-2\mu t} |x(0)|_2^2.
\end{align}
The constant $\mathscr{C}$ can be chosen as
\begin{align}\label{eq:constantM}
\mathscr{C}:= \begin{cases}
(\lambda_{\min}^{P})^{-1} \lambda_{\max}^{P},& M=1,\\
2(\lambda_{\min}^{P(0)})^{-1} \lambda_{\max}^{P(0)}c_M\,
\displaystyle\max_{n\in I_\mu}
\Big[\sum_{m=1}^{l_n}\frac{\beta_n^m}{\displaystyle\min_{k\in\{1,\ldots,m\}}\beta_n^k} \Big],& M\geq 2,
\end{cases}
\end{align}
where $\lambda_{\min}^{P(0)}$ is the smallest and $\lambda_{\max}^{P(0)}$ is the largest eigenvalue of the matrix $P(0)$ (with $P(0)\equiv P$ for $M=1$), which is defined in \eqref{eq:defP}. The constants $c_M$ for $M\geq 2$ are given as:
\begin{align*}
c_{M}&= 2^{M-2} \left(\prod_{j=1}^{M-1} 4j^2-1\right) \left(\prod_{j=2}^{M} \sum_{k=1}^j \frac{1}{[(j-k)!]^2}\right).
\end{align*}
\end{theorem}
The technical proof of this result is deferred to Appendix \ref{pr:EuclidDecay}.
\begin{remark}
If, for all $n\in I_\mu$, the weights $\beta_n^m$ of the matrix $P(0)$ are monotonically decreasing in $m$, i.e.\ $\beta_n^m \geq \beta_n^{m+1}$ for $m\in\{1,\ldots,l_n-1\}$, then
\begin{align*}
\displaystyle\max_{n\in I_\mu}
\Big[\sum_{m=1}^{l_n}\frac{\beta_n^m}{\displaystyle\min_{k\in\{1,\ldots,m\}}\beta_n^k} \Big]=M.
\end{align*}
\end{remark}
\begin{remark}
Note that we could calculate the solution to the ODE system \eqref{eq:ODE} directly, by means of its Jordan transformation,  and get qualitatively the same decay behavior as in \eqref{eq:EuclidEst}, but possibly with a different multiplicative constant: Using $x(t) = e^{-Ct}x(0)$, $C = (V^H)^{-1}J^H V^H$, we obtain
\begin{align}
|x(t)|_2^2 &\leq |V|^2_2|V^{-1}|^2_2 |e^{-J^Ht}|^2_2|x(0)|^2_2 \nonumber\\
&\leq|V|^2_2|V^{-1}|^2_2 \hat{c}_M (1+t^{2(M-1)}) e^{-2\mu t}|x(0)|^2_2,\label{eq:Jconstant}
\end{align}
where $|V|_2$ denotes the matrix norm induced by the vector norm $|\cdot|_2$, $V$ is the transformation matrix from \eqref{eq:V}, $J:=\diag(J_1,\ldots,J_N)$ the corresponding Jordan matrix, and $\hat{c}_M$  depends only on the largest Jordan block.
So what is the gain of the result in Theorem \ref{th:EuclidDecay}?

Firstly, the construction of the matrix $P(t)$ and the method of estimating the $P(t)$-norm decay of the solution can be translated almost directly to the infinite dimensional setting of the Fokker--Planck equation with linear drift, where a direct way of calculating the decay (as in the finite dimensional ODE case) is not possible (see \cite{Arnold2014} for the exponential decay in the non-defective case and \cite{Monmarche2015} for an improved decay in the defective case). In the Fokker--Planck setting on $\R^d$, the place of the $P(t)$-norm is taken by the modified Fisher information involving $P(t)$.

Secondly, the result also makes it possible to systematically calculate the multiplicative constant $\mathscr{C}$ from \eqref{eq:EuclidEst}, which we will use in \S \ref{sec:hc}--\S \ref{sec:fpe}, and which can be further exploited to get decay results for infinite dimensional ODE systems (see \cite{Achleitner2016}, \S 4.3).
\end{remark}

\subsection{Uniform decay estimates in non-defective limits}
The advantage of the $P(t)$-norm estimation compared to \eqref{eq:Jconstant} can be seen in the following example.

\begin{example}\label{ex:Def1toDef0}
Consider the matrix
\begin{align*}
C_\epsilon :=
\begin{pmatrix}
1~& \epsilon\\
0~& 1
\end{pmatrix}
\end{align*}
with arbitrary $\epsilon\neq 0$. Its corresponding Jordan transformation matrix reads
\begin{align*}
V_\epsilon:=
\begin{pmatrix}
1~& 0 \\
0~& \frac{1}{\epsilon}
\end{pmatrix},
\end{align*}
and $M=2$.
For $\epsilon\to 0$ the factor $|V_\epsilon|_2 |V_\epsilon^{-1}|_2$ in \eqref{eq:Jconstant} becomes unbounded of order $\epsilon^{-2}$ (even though the true decay of the solution improves to $e^{-t}|x(0)|_2$ in the limit). This is due to the discontinuity of the Jordan transformation at the transition from defectiveness to non-defectiveness. We apply Theorem \ref{th:EuclidDecay} to the ODE system $\dot{x}=-C_\epsilon x$, with the following eigenvectors of $C_\epsilon^H$:
\begin{align*}
v^{(0)}_1 =
\begin{pmatrix}
0,& 1
\end{pmatrix}^T\quad \text{ and }\quad
v^{(1)}_1 =
\begin{pmatrix}
\frac{1}{\epsilon},& 0
\end{pmatrix}^T.
\end{align*}
When using $\beta_{1,\epsilon}^1 = 1$ and $\beta_{1,\epsilon}^2 = \epsilon^2$ in \eqref{eq:defP}, we get $P_\epsilon(0)=I$, and hence
the constant
\begin{align}\label{eq:constantDef1}
\mathscr{C}_\epsilon=12\cdot \max\{2,1+\epsilon^2\}
\end{align}
stays bounded in the \emph{non-defective limit} $\epsilon\to 0$.\hfill$\Diamond$
\end{example}

This example shows that, while the method presented here is still relying on the Jordan transformation (as $P(t)$ is constructed with generalized eigenvectors), the additional weights $\beta_n^m$ in the matrix $P(t)$ allow for estimates more closely related to the actual behavior of the solutions. As sketched in Example \ref{ex:Def1toDef0} this can allow for (but does not guarantee) an estimate that is uniform in the \emph{non-defective limit}. In \S \ref{sec:hc}--\S \ref{sec:fpe} we will see further examples of specific choices of the weights $\beta_n^m$ that lead to uniform estimates in the non-defective limit.

\begin{remark}\label{rem:decayest}
While the proof of Theorem \ref{th:EuclidDecay} is formulated to work for arbitrarily large defects, more careful estimations can improve the decay estimate. We shall now show an improvement for defect one. For any $n\in I_\mu$ the inequality \eqref{eq:Pm0estimate1} with $m=2$ and $\theta=\frac12$ yields
\begin{align*}
|x(t)|_{P_n^2(0)}^2 \leq 2 e^{-2\mu t}\left( t^2 |x(0)|^2_{P_n^1(0)} +  |x(0)|_{P_n^2(0)}^2\right),
\end{align*}
where we used \eqref{eq:case3} with $P_n^1(t)=P_n^1(0)$.

For $n\in I_\mu$ with $l_n=2$ the decay estimate for $P_n(0)=\beta^1_n P^1_n(0)+\beta^2_n P^2_n(0)$ follows as
\begin{align*}
|x(t)|^2_{P_n(0)} \leq 2 e^{-2\mu t} (1 + \frac{\beta_n^2}{\beta^1_n}t^2) |x(0)|^2_{P_n(0)}.
\end{align*}
We can use this estimate to get an improved upper bound for solutions from Example \ref{ex:Def1toDef0} in the Euclidean norm (compared to \eqref{eq:EuclidEst} with \eqref{eq:constantDef1}): With the same matrix choice $P_\epsilon(0)=P^1_1(0) + \epsilon^2 P^2_1(0)=I$ as in Example \ref{ex:Def1toDef0}, it follows that solutions to $\dot{x}=-C_\epsilon x$ satisfy
\begin{align}\label{eq:epsPrecise}
|x(t)|_2^2&=|x(t)|_{P_\epsilon(0)}^2 \leq 2 e^{-2t} (1 + \epsilon^2 t^2) |x(0)|^2_2,\quad t\geq 0.
\end{align}
This estimate not only yields a bounded multiplicative constant for $\epsilon\to 0$, but also yields a sharp decay rate --- namely purely exponential --- for the non-defective limit case $\epsilon=0$. In comparison, the solution propagator norm estimate for all $\epsilon\in\R$ is given as
\begin{align*}
|e^{-C_\epsilon t}x(0)|_2^2 &= e^{-2 t}\left|
\begin{pmatrix}
1& -\epsilon t\\
0&~1
\end{pmatrix}
\right|^2_2|x(0)|^2_2\\
&=e^{-2t} \left(1 + \frac{\epsilon^2 t^2}{2} + \sqrt{\epsilon^2 t^2 + \frac{\epsilon^4t^4}{4}}\right)|x(0)|^2_2 \numberthis \label{eq:evpropest}\\
&\overset{\epsilon t \to \infty}{\approx} e^{-2t}(2+\epsilon^2 t^2)|x(0)|^2_2.
\end{align*}
This shows that \eqref{eq:epsPrecise} is rather accurate.
\end{remark}

\begin{remark}\label{rem:Ptime}
Consider an ODE \eqref{eq:ODE} with matrix $C$ that has a Jordan block $J_{n_2}$ in Case 2 from \S \ref{subsec:pmatrix}, i.e. $l_{n_2}>1$ and $\re(\lambda_{n_2})>\mu$.
Then, the construction of Case 2 can be replaced with the one of Case 3. This means, exchanging the time-constant matrix $P_{n_2}$, which has predetermined weights $b_{n_2}^m$, by a time-dependent matrix $\wt{P}_{n_2}(t)$, which allows for arbitrary weights $\beta_{n_2}^m>0$. For simplicity let us assume $M=1$. With the appropriate (straight forward) modifications of Lemma \ref{lem:Pexists} and Theorem \ref{th:EuclidDecay} (treating $J_{n_2}$ as Case 3), this yields the decay estimate \eqref{eq:EuclidEst} with the modified constant
\begin{align}
\wt{\mathscr{C}}: = 2(\lambda_{\min}^{\wt{P}(0)})^{-1} \lambda_{\max}^{\wt{P}(0)}c_{l_{n_2}}\,
\Big[\sum_{m=1}^{l_{n_2}}\frac{\beta_{n_2}^m}{\displaystyle\min_{k\in\{1,\ldots,m\}}\beta_{n_2}^k} \Big],
\end{align}
where $\wt{P}(t)$ is the matrix defined by \eqref{eq:defP}, but with $\wt{P}_{n_2}(t)$ instead of $P_{n_2}$. 

For ODE families $\dot{x}_\epsilon=-C_\epsilon x_\epsilon$ and their non-defective limits $\epsilon\to 0$, this modification can be beneficial: The additional weights of $\wt{P}_{n_2}(t)$ provide further possibilities to obtain a multiplicative constant $\wt{\mathscr{C}}_\epsilon$ that is bounded for $\epsilon\to 0$. In \S \ref{subsec:FPEdecay}  (Case $k=3$), we will see an example of an ODE family where Theorem \ref{th:EuclidDecay} yields an unbounded constant $\mathscr{C}_\epsilon$ (for all possible weights) but a bounded constant $\wt{\mathscr{C}}_{\epsilon}$ for $\epsilon\to 0$ (with the correct choice of weights).
\end{remark}

\subsection{Uniform decay for a family of ODEs}
We shall consider now an extension of Example \ref{ex:Def1toDef0} which will be relevant for the PDEs discussed in \S \ref{sec:hc}--\S \ref{sec:fpe}: We consider the matrix family with \mbox{parameter $z\in\R$}
\begin{align}\label{eq:ODEfam}
C(z): =
\begin{pmatrix}
\mu(z)& \mu'(z)\\
0& \mu(z)
\end{pmatrix}
=
\mu(z)
\begin{pmatrix}
1& ~~\frac{\mu'(z)}{\mu(z)}\\
0& ~~1
\end{pmatrix}
\end{align}
with a given function $\mu\in C^1(\R)$ and $\mu(z) \geq \mu_{\min} = \mu(z_0) >0$. For simplicity let $z_0\in\R\cup \{\infty,-\infty\}$ be the unique global point of minimum (infimum if $|z_0|=\infty$) of $\mu$.

We are now interested in a uniform-in-$z$ estimate on the matrix propagator $e^{-C(z) t}$ with $t\geq 0$, based on the estimate \eqref{eq:epsPrecise}. To this end we have to consider the interplay of two effects: On the one hand the parameter value $z=z_0$ yields the smallest exponential decay rate $\mumin$ but it is without defect, since $\mu'(z_0)=0$ makes $C(z_0)$ a diagonal matrix.
On the other hand the parameters $z\neq z_0$ yield a larger decay, but with a defect (as long as $\mu'(z)\neq 0$). Hence we shall be interested in the question, whether or not the typical defective decay of the form $\mathcal{O}((1+t^2) e^{-2\mu_{\min} t})$ persists for the uniform estimate of $|e^{-C(z) t}|^2$ for $t\to+\infty$. In the subsequent examples we shall illustrate that both scenarios are in fact possible.
\begin{example}\label{ex:unidefect}
Let $\mu(z) := \mumin + \alpha z^2$ with some $\alpha>0$, and hence $z_0=0$. From Example \ref{ex:Def1toDef0} with \eqref{eq:epsPrecise} (using $\epsilon = \frac{\mu'(z)}{\mu(z)}$, $t\mapsto t \mu(z)$) we obtain the uniform decay estimate
\begin{align}\label{eq:exQuad}
|e^{-C(z) t} |^2_2 \leq 2 e^{-2 \mumin t} \sup_{z\in\R} f_1(z,t),\quad  z\in\R,t\geq 0,
\end{align}
with $$f_1(z,t):= (1 + 4\alpha^2 z^2 t^2)e^{-2 \alpha z^2 t}.$$
An elementary computation yields
\begin{align*}
 \sup_{z\in\R} f_1(z,t)  =
 \begin{cases}
1,& \alpha t \leq \frac12,\\
2\alpha t e^{-\frac{2\alpha t-1}{2\alpha t}},& \alpha t > \frac12,
\end{cases}
\end{align*}
with the asymptotic behavior $\sup_{z\in\R} f_1(z,t) = \mathcal{O}(\frac{2\alpha}{e}t)$ as $t\to +\infty$. Hence, estimate \eqref{eq:exQuad} exhibits the typical defective decay behavior, and the term $t e^{-2\mumin t}$ cannot be dropped in the estimate.\hfill$\Diamond$
\end{example}
\begin{example}\label{ex:noalg}
Let $\mu(z) := \mu_0 + \alpha e^{\beta z}$ with some $\alpha>0$ and $\beta\in\R\setminus\{0\}$ with $|\beta|<2$. Here, $z_0=-\infty\operatorname{sgn}(\beta)$. Then \eqref{eq:epsPrecise} yields the uniform decay estimate
\begin{align}\label{eq:exExp}
|e^{-C(z) t} |^2_2 \leq 2 e^{-2 \mu_0 t} \sup_{z\in\R} f_2(z,t),\quad  z\in\R,t\geq 0,
\end{align}
with
\begin{align*}
f_2(z,t) : = (1+\alpha^2 \beta^2 e^{2 \beta z} t^2) e^{-2\alpha e^{\beta z} t}.
\end{align*}
Since $\partial_t f_2(z,0) = -2 \alpha e^{\beta z}<0$ and $\partial_t f_2(z,t) \neq 0$ for every $z\in\R$, $t\geq 0$, we conclude
\begin{align*}
|e^{-C(z) t} |^2_2 \leq 2 e^{-2\mu_0 t} \sup_{z\in\R} f_2(z,0) = 2 e^{-2 \mu_0 t},\quad z\in\R, t\geq 0.
\end{align*}
Hence, this example shows a purely exponential decay behavior, which is rather typical for the non-defective case.\hfill$\Diamond$
\end{example}
We remark that we could not find an example of a parameter function $\mu\in C^1(\R)$ with a minimum at $|z_0|<\infty$ for which the algebraic factor vanishes in the uniform estimate.

In the following sections \S \ref{sec:hc}--\S\ref{sec:fpe} we investigate parabolic and kinetic evolution equations in which equation coefficients depend on an uncertainty variable $z\in\R$. After a Fourier decomposition, the sensitivity analysis leads to families of defective ODE systems of type similar to \eqref{eq:ODEfam}, for which we are interested in uniform-in-$z$ decay estimates of solutions with sharp rate. Sharpness is understood here in the sense that in the class of $C^1(\R)$ parameter functions $\mu(z)$ satisfying $\mu_0:=\inf_{z\in\R}\mu(z)>0$, the decay estimate is of type $\mathcal{O}(1+t^m)e^{-2\mu_0 t}$ for large $t$, with minimal $m\in\N_0$. As Example \ref{ex:unidefect} illustrates, $m>0$ is necessary to cover arbitrary $C^1(\R)$ parameter functions.

With the three examples of \S 3--\S 5 we shall illustrate the various challenges of this procedure to obtain estimates: \emph{uniformity in the Fourier modes and in the non-defective limit(s).}

\section{Linear convection-diffusion equations with uncertain coefficients}\label{sec:hc}
First we consider the parabolic equation on the 1D torus
\begin{align}\label{eq:heat}
\partial_t u(x,z,t)&= -a(z)\partial_x u(x,z,t)+b(z)\partial^2_{x}u(x,z,t),\quad x\in\T^1, t\geq 0,\\
u(x,z,0)&=u^0(x,z),
\end{align}
for $u(x,z,t)\in\R$ with the space variable $x$, convection coefficient $a(z)\in\R$ and diffusion coefficient $b(z)$ satisfying $b_0:=\inf_{z\in\R} b(z)>0$. We are interested in the sensitivity of solutions with respect to the uncertainty parameter $z\in\R$ contained in the coefficients. We assume that the coefficients satisfy $a,b\in C^1(\R)$. For each $z\in\R$, the equation is  mass conserving (in time), i.e.\ $\frac{1}{2\pi}\int^{2\pi}_0 u(x,z,t)dx = const.$, and the unique normalized steady state is given as $u^\infty(x,z)=1$. Correspondingly we shall also assume that the initial condition is normalized as
$\frac{1}{2\pi}\int_0^{2\pi} u^0(x,z) dx = 1$ for all $z\in\R$.

A Fourier expansion of $u$ with respect to $x\in\T^1$, allows to rewrite the PDE (for each fixed $z$) as a family of ODEs. With the notation $u(x,z) = \sum_{k\in\Z} u_k(z)e^{ikx}$, the equation for each Fourier mode $u_k$, $k\in\Z$ reads
\begin{align}\label{eq:heatFourier}
\partial_t u_k(z) = -ik a(z) u_k(z)- k^2 b(z) u_k(z),
\end{align}
with the explicit solutions $u_k(z,t) = e^{-k^2 b(z) t - ik a(z)t}u_k(z,0)$. Due to the above normalization we have for the $k=0$ mode: $u_0(z,t)=u^\infty_0(z)=1$.
\subsection{First order parameter sensitivity analysis}\label{subsec:hcSens1}
Now we analyze the (linear order) sensitivity of the equation with respect to the uncertainty in the coefficients $a(z)$ and $b(z)$. Therefore we consider the evolution equation for $v(x,z,t):=\partial_z u(x,z,t)$, given as
\begin{align}\label{eq:sensheat}
\partial_t v(z) =  -(\partial_z a(z)) \partial_x u(z)+ (\partial_z b(z))\partial_{x}^2 u(z)   - a(z) \partial_x v(z)+ b(z) \partial_x^2 v(z) .
\end{align}
The Fourier modes $v_k(z,t):=\partial_z u_k(z,t)$ for $k\in\Z$ satisfy
\begin{align}\label{eq:sensheatFourier}
\partial_t v_k(z) =  -ik(\partial_z a(z)) u_k(z)- k^2 (\partial_z b(z))u_k(z)   - ik a(z) v_k(z)- k^2 b(z) v_k .
\end{align}
For $k\in\Z\setminus\{0\}$ the system of \eqref{eq:heatFourier} and \eqref{eq:sensheatFourier} reads
\begin{align}\label{eq:firstorderHeat}
\partial_t
\underbrace{\begin{pmatrix}
u_k
\\v_k
\end{pmatrix}}_{\udef{y_k(z,t):=}}
= -k^2\underbrace{\begin{pmatrix}
 b(z) + \frac{ia(z)}{k}& 0\\
\partial_z b(z) + \frac{i \partial_z a(z)}{k}~~&  b(z) + \frac{ia(z)}{k}
\end{pmatrix}}_{\udef{C_k(z):=}}
\begin{pmatrix}
u_k
\\v_k
\end{pmatrix}.
\end{align}
Our goal is to obtain a decay estimate with sharp decay rate for solutions to \eqref{eq:firstorderHeat}, uniform in $z\in\R$ by applying Theorem \ref{th:EuclidDecay}.

Due to the normalization of the initial conditions $u^0$, we have $v_0(z,t) = \frac{1}{2\pi} \int^{2\pi}_0 v(x,z,t) dx \equiv 0$, and in particular for the initial condition $\int_0^{2\pi}v^0(x,z)dx = 0$. Hence, its steady state is $v_0^\infty(z)=0$; the (expected) decay of all higher modes $v_k, k\in\Z\setminus\{0\}$ implies $v^\infty(x,z)\equiv 0$. With the notation $y:=(u,v)^T$, the unique (normalized) steady state of the system \eqref{eq:heat}, \eqref{eq:sensheat} is $y^\infty(x,z) \equiv (1, 0)^T$.

For each Fourier mode $k\in\Z\setminus\{0\}$, the double eigenvalue of the matrix $C_k(z)$ is given as
\begin{align*}
\lambda_k(z):= b(z) + \frac{i a(z)}{k}.
\end{align*}
Hence, the matrix is positive stable and the steady state is given as $y_k^\infty = 0\in\C^2$. The spectral gap of the evolution operator in \eqref{eq:firstorderHeat} is given by $\mu_k(z):=k^2 b(z)>0$ and the eigenvalue $\lambda_k(z)$ of $C_k(z)$ is defective, if and only if $\partial_z \lambda_k(z) \neq 0$.

If we consider all Fourier modes, the spectral gap for the whole sequence $\{y_k(z)\}_{k\in\Z\setminus\{0\}}$ is given as
\begin{align*}
\mu(z):=\min_{k\in\Z\setminus\{0\}} \mu_k(z) =\mu_{\pm 1}(z)= b(z),
\end{align*}
and it is realized by the modes $k=\pm 1$. The steady state is given by the sequence $\{y_k^\infty\}_{k\in\Z} = \{(\delta_{0k},0)^T\}_{k\in\Z}$, with $\delta_{0k}$ denoting the Kronecker delta.

In what follows, we investigate the decay rate of solutions to the system of equations \eqref{eq:heat} and \eqref{eq:sensheat} towards the steady state with a sharp rate, uniform in the uncertainty variable $z$.

The solution vector $y(\cdot,z,t)\in L^2(0,2\pi)$ is equivalent to $\{y_k(z,t)\}_{k\in\Z}$ by Parseval's identity. As the ODE system for each Fourier mode $k\in\Z\setminus\{0\}$ can be defective for certain values in $z\in\R$, we expect a uniform-in-$z$ decay rate that is not purely exponential. In fact, for non-constant $a$ and $b$, defectiveness is the more typical behavior of the matrix $C_k(z)$.

For each fixed $k\in\Z\setminus\{0\}$, we proceed with a case distinction between the defective and non-defective case.

\subsubsection{Case 1; $z\in\R$, such that $\partial_z\lambda_k(z)=0$:}
The matrix $C_k(z)$ is diagonal and the solutions are given as $y_k(z,t)=e^{-k^2\lambda_k(z)t} y_k(z,0)$. Hence, for each $k$ and $z$, the decay of solutions to \eqref{eq:firstorderHeat} is given as
\begin{align}\label{eq:HeatModeDecayCase1}
|y_k(z,t)|^2_2 = e^{-2k^2b(z) t} |y_k(z,0)|^2_2.
\end{align}

\subsubsection{Case 2; $z\in\R$, such that $\partial_z\lambda_k(z)\neq 0$:} In this case the eigenvalue $\lambda_k(z)$ of $C_k(z)$ is defective of order 1, i.e.\ $M=2$, for $k\in\Z\setminus\{0\}$. Hence we shall apply Theorem \ref{th:EuclidDecay} to get a sharp decay estimate for solutions to \eqref{eq:firstorderHeat}.

To denote the dependence on the Fourier mode $k\in\Z$, we shall use the subscript $k$, e.g., we use $P_k(z,0)$ for the matrix $P(z,0)$, defined in \eqref{eq:defP}, that corresponds to $C_k(z)$. Similarly, we denote $P_k(z,0)$'s weights $\beta_n^m$ in \eqref{eq:defP} by $\beta_{n,k}^m$.

For $k\in\Z\setminus\{0\}$, the matrix $C_k(z)$ is positive stable and the eigenvector and generalized eigenvector of $C^H_k(z)$, corresponding to $\overline{\lambda}_k(z)$, are given as
\begin{align*}
&v_{1,k}^{(0)} =
\begin{pmatrix}
1,&0
\end{pmatrix}^T,&
v_{1,k}^{(1)}(z)=
\begin{pmatrix}
0,&\frac{1}{\partial_z \overline{\lambda}_k(z)}
\end{pmatrix}^T.
\end{align*}
In analogy to Example \ref{ex:Def1toDef0} we choose $\beta_{1,k}^1 = 1$ and $\beta_{1,k}^2(z) = |\partial_z\lambda_k(z)|^2$, leading to
\begin{align*}
P_k(z,0) &= v_{1,k}^{(0)} \otimes v_{1,k}^{(0)} + |\partial_z\lambda_k(z)|^2 v_{1,k}^{(1)} \otimes v_{1,k}^{(1)}= I.
\end{align*}
An appropriate choice of $\beta_{1,k}^m$ is essential here, in order to make the matrix $P_k(z,0)$ (and hence the constant $\mathscr{C}_k(z)$ below) uniformly bounded for $\partial_z\lambda_k(z)\to 0$, i.e. for the \textit{non-defective limit}.

Now we can apply Theorem \ref{th:EuclidDecay} (for the rescaled time $\tau_k=k^2t$) to get the following decay estimate for solutions to the system \eqref{eq:firstorderHeat} for each $k\in\Z\setminus\{0\}$ and $z\in\R$:
\begin{align}\label{eq:HeatModeDecayCase2}
|y_k(z,t)|_2^2 \leq \mathscr{C}_k(z) (1+k^4t^2) e^{-2k^2 b(z)t}|y_k(z,0)|^2_2,
\end{align}
with the constant $\mathscr{C}_k(z)\geq 0$ defined in \eqref{eq:constantM}, given as
\begin{align}\label{eq:Case2const}
\mathscr{C}_k(z)&=12 \cdot \left(1+\frac{|\partial_z \lambda_k(z)|^2}{\min\{1,|\partial_z \lambda_k(z)|^2\}}\right)=
12\cdot \max\{2,1+|\partial_z\lambda_k(z)|^2\}.
\end{align}

Combining Case 1 and Case 2, we infer a decay estimate for the first order parameter sensitivity equations by applying Parseval's identity:
\begin{theorem}\label{th:sens1Heat}
Let $a,b\in C^1(\R)$ where $b_0:=\inf_{z\in\R}b(z)>0$ and $\partial_z a, \partial_z b\in L^\infty(\R)$. Then, there exists a constant $\mathscr{C}>0$, such that normalized  solutions $y(x,z,t)=(u(x,z,t),v(x,z,t))^T$ of the system \eqref{eq:heat}, \eqref{eq:sensheat} with steady state $y^\infty := (1,0)^T$ satisfy
\begin{align*}
\sup_{z\in\R}\|y(\cdot,z,t)-y^{\infty}\|^2_{L^2(0,2\pi;\R^2)} \leq \mathscr{C}(1+t^2)e^{-2b_0 t} \sup_{z\in\R}\|y(\cdot,z,0)-y^\infty \|^2_{L^2(0,2\pi;\R^2)}
\end{align*}
for $t\geq 0$.
\end{theorem}
\begin{proof}
Combining both estimates \eqref{eq:HeatModeDecayCase1} and \eqref{eq:HeatModeDecayCase2} leads to
\begin{align}
|y_k(z,t)|_2^2 \leq \wt{\mathscr{C}} (1+k^4t^2) e^{-2k^2 b(z)t}|y_k(z,0)|^2_2,\quad k\in\Z\setminus\{0\},
\end{align}
with the constant
\begin{align*}
\wt{\mathscr{C}}= \sup_{k\neq 0,z\in\R}\mathscr{C}_k(z) \leq 12 \max\{2, 1+\|\partial_z a\|_\infty^2 + \|\partial_z b\|^2_\infty\}
\end{align*} independent of $z\in\R$ and $k\in\Z\setminus\{0\}$. With Parseval's identity we obtain
\begin{align*}
\|y(\cdot,z,t) - y^\infty \|_{L^2(0,2\pi;\R^2)}^2 &= \frac{1}{2\pi}\sum_{k\in\Z} |y_k(z,t)-y^\infty_k|_2^2\\
& \leq \frac{1}{2\pi}\sum_{k\in\Z\setminus\{0\}} \wt{\mathscr{C}} (1+k^4t^2)e^{-2k^2 b(z)t}|y_k(z,0)|^2_2 \\
&\leq \mathscr{C} (1+t^2)e^{-2b_0 t} \|y(\cdot,z,0) - y^\infty \|_{L^2(0,2\pi;\R^2)}^2,
\end{align*}
where we used the estimate
\begin{align*}
(1+k^4t^2)e^{-2k^2b_0 t} \leq c (1+t^2)e^{-2 b_0 t},\quad t\geq 0, k\neq 0,
\end{align*}
with $c:=\max_{t\geq 0} (1+t^2)e^{-2 b_0 t}$.
Taking the supremum over $z\in\R$ completes the proof.\qed
\end{proof}
\subsection{Second order parameter sensitivity analysis}\label{subsec:HCsecond}
Next we shall extend the above analysis to second order. This will also illustrate the challenges involved in obtaining uniform decay estimates in defective limits.

We assume $a,b\in C^2(\R)$ and denote $w(x,z,t):=\partial^2_z u(x,z,t)$. By differentiation of \eqref{eq:sensheat} with respect to $z$, the second order sensitivity equation is given as
\begin{align}\label{eq:sens2heat}
\begin{aligned}\partial_t w(z)&= -(\partial_z^2 a(z)) \partial_x u(z) + (\partial^2_z b(z))\partial_x^2 u(z)\\
&\qquad- 2(\partial_z a(z)) \partial_xv(z) +2 (\partial_zb(z))\partial^2_x v(z)\\
&\qquad - a(z)\partial_x w(z) + b(z) \partial^2_x w(z).
\end{aligned}
\end{align}
The system for the Fourier mode $k\in\Z\setminus\{0\}$ of $(u,v,w)^T$, with $w_k(z,t):=\partial_z v_k(z,t)$, is given as
\begin{align}\label{eq:secorderODE}
\partial_t
\underbrace{\begin{pmatrix}
u_k\\
v_k\\
w_k
\end{pmatrix}}_{\udef{y_k(z,t):=}}
=
-k^2
\underbrace{\begin{pmatrix}
\lambda_k(z)~~& 0 & 0\\
\partial_z\lambda_k(z)~~& \lambda_k(z) & 0\\
\partial^2_z\lambda_k(z)~~& 2\partial_z\lambda_k(z)~~ & \lambda_k(z)
\end{pmatrix}}_{\udef{D_k(z):=}}
\begin{pmatrix}
u_k\\
v_k\\
w_k
\end{pmatrix}.
\end{align}
As before, $w_0(z,t)\equiv 0$, and in particular for the initial condition $\int_0^{2\pi}w^0(x,v)dx=0$. Hence the unique (normalized) steady state of the second order sensitivity system is $y^\infty(x,z):=(u^\infty(x,z), v^\infty(x,z), w^\infty (x,z))^T \equiv (1, 0, 0)^T$.

The triple eigenvalue is $\lambda_k(z)$, with $\re(\lambda_k(z))>0$ and its defectiveness for $k\in\Z\setminus\{0\}$ depends on the values of $\partial_z\lambda(z)$ and $\partial_z^2\lambda(z)$, i.e.\
\begin{align*}
\rank( D_k(z)-\lambda_k(z)) =
\begin{cases}
0,& \text{if }\partial_z\lambda_k(z)=\partial_z^2\lambda_k(z) = 0,\\
1,& \text{if }\partial_z\lambda_k(z) = 0\, \& \,\partial^2_z\lambda_k(z) \neq 0,\\
2,& \text{if }\partial_z\lambda_k(z) \neq 0.
\end{cases}
\end{align*}

As in the first order analysis, we need to discuss the decay behavior of these three cases separately:
\subsubsection*{Case 1; $z\in\R$ such that $\partial_z \lambda(z) = \partial_z^2 \lambda(z) = 0$:}
In this case, the eigenvalue $\lambda_k(z)$ is non-defective and the solutions are given as $y_k(z,t) = e^{-k^2 \lambda_k(z)t} y_{k}(z,0)$ from which we obtain the decay
\begin{align} \label{eq:sens2case1}
|y_k(z,t)|^2_2 = e^{-2k^2b(z) t} |y_k(z,0)|^2_2.
\end{align}
\subsubsection*{Case 2; $z\in\R$, such that $\partial_z \lambda(z) =0$ and $ \partial_z^2 \lambda(z) \neq 0$:}
In this case the eigenvalue $\lambda_k(z)$ is defective of order one, i.e. $M=2$. To obtain a sharp decay estimate of solutions, we construct $P_k(z,0)$ according to \eqref{eq:defP}: $N=2$, $l_1=2$, $l_2=1$, $M=2$ and choosing $\beta_{1,k}^1 = 1$, $\beta_{1,k}^2 = |\partial_z^2 \lambda_k(z)|^2$ and $\beta_{2,k}^1=1$. The (generalized) eigenvectors of $D_k^H(z)$ are given as
\begin{align*}
&v_{1,k}^{(0)} =
\begin{pmatrix}
1,&0,& 0
\end{pmatrix}^T,&
v_{1,k}^{(1)}(z)=
\begin{pmatrix}
0,& 0, \frac{1}{\partial^2_z \overline{\lambda}_k(z)}
\end{pmatrix}^T,&
&v_{2,k}^{(0)} =
\begin{pmatrix}
0,&1,& 0
\end{pmatrix}^T.
\end{align*}
This leads to
\begin{align*}
P_k(z,0) &=v^{(0)}_{1,k}\otimes v^{(0)}_{1,k} + |\partial^2_z\lambda_k(z)|^2 v^{(1)}_{1,k}(z)\otimes v^{(1)}_{1,k}(z)+ v^{(0)}_{2,k}\otimes v^{(0)}_{2,k}=I.
\end{align*}
We can now apply Theorem \ref{th:EuclidDecay} (for the rescaled time $\tau_k = k^2 t$) and get the decay estimate
\begin{align}\label{eq:sens2case2}
|y_k(z,t)|^2_2 \leq \mathscr{C}_k(z) (1+k^4t^2) e^{-2 k^2 b(z) t} |y_k(z,0)|^2_2
\end{align}
with the the constant $\mathscr{C}_k(z)$, defined in \eqref{eq:constantM}, given as
\begin{align*}
\mathscr{C}_k(z) = 12\cdot \max\{2,1+|\partial^2_z\lambda_k(z)|^2\}.
\end{align*}
Note that this constant $\mathscr{C}_k(z)$ is uniformly bounded in the \textit{non-defective limit} $\partial_z^2 \lambda_k(z)\to 0$ (from defect 1 to non-defective), but it does not reduce to \eqref{eq:sens2case1}, hence it is not uniformly sharp.
\subsubsection*{Case 3; $z\in\R$, such that $\partial_z \lambda_k(z)\neq 0$:}
The eigenvalue $\lambda_k(z)$ is defective of order two with $N=1$, $l_1=3$ and $M=3$. The (generalized) eigenvectors of $D_k^H(z)$ are given as
\begin{align*}
&v^{(0)}_{1,k}(z)=
\begin{pmatrix}
1,&0,&0
\end{pmatrix}^T, \quad
v^{(1)}_{1,k}(z)=
\begin{pmatrix}
0,& \frac{1}{\partial_z\overline{\lambda}_k(z)},&0
\end{pmatrix}^T,\\
&v^{(2)}_{1,k}(z)=
\begin{pmatrix}
0,&\frac{-\partial^2_z\overline{\lambda}_k(z)}{2(\partial_z\overline{\lambda}_k(z))^3} ,&\frac{1}{2(\partial_z\overline{\lambda}_k(z))^2}
\end{pmatrix}^T.
\end{align*}

For this case the previous strategy of finding weights for $P(t)$ (as defined in \eqref{eq:defP}) that give a uniform in $z$ decay estimate for solutions via Theorem \ref{th:EuclidDecay} does not work directly. All choices of weights $\beta_{1,k}^j(z)$ for $j\in\{1,2,3\}$ lead to constants $\mathscr{C}_k(z)$ (defined in \eqref{eq:constantM}) that are not bounded uniformly in $z$. The problem arises for the defective limit from defect 2 to defect 1, more precisely, for sequences $(z_n)_{n\in\N}\subset\R$ such that $0\neq\partial_z \lambda_k(z_n)\to 0$ in combination with $\frac{\partial^2_z\lambda_k(z_n)}{\partial_z \lambda_k(z_n)}\not\to 0$ as $n\to \infty$. All weight choices of $\beta_{1,k}^j(z_n)$ lead to $(\lambda_{\min}^{P_k(z_n,0)})^{-1} \lambda_{\max}^{P_k(z_n,0)}\to \infty$ due to the three different powers of $\partial_z \overline{\lambda}_k(z_n)$ that appear in the (generalized) eigenvectors of $D_k^H(z_n)$.

However, this problem can be fixed with small adjustments to the proof of \mbox{Theorem \ref{th:EuclidDecay}}, which yield a uniform in $z$ decay estimate.

Define
\begin{align*}
\widetilde{w}^3_k(z,t) &:= w^3_{1,k}(z,t) + \frac{\partial_z^2 \overline{\lambda}_k(z)}{2(\partial_z \overline{\lambda}_k(z))^2}w^2_{1,k}(z,t),
\end{align*}
which is our replacement for $w_k^3(z,t)$, with $w^j_{1,k}(z,t)$ for $j\in\{2,3\}$ from \eqref{eq:wnm}. It satisfies
\begin{align*}
\widetilde{w}^3_k(z,0)=\begin{pmatrix}
0,&0,&\frac{1}{2(\partial_z\overline{\lambda}_k(z))^2}
\end{pmatrix}^T,
\end{align*}	 which eliminates the problematic factor ${2(\partial_z\overline{\lambda}_k(z))^3}$ present in $w^3_k(z,0)$.
The corresponding semi-norm matrix is
\begin{align*}
\widetilde{P}^3_k(z,t) := \wt{w}_k^3(z,t) \otimes \wt{w}_k^3(z,t).
\end{align*}
Similar to definition \eqref{eq:defP}, let
\begin{align}\label{eq:defPtilde}
\wt{P}_k(z,t):= P^1_{1,k}(z,t) +|\partial_z \lambda_k(z)|^2 P^2_{1,k}(z,t) + 4|\partial_z \lambda_k(z)|^4 \wt{P}_{k}^3(z,t),
\end{align}
which is positive definite for all $k\neq 0$, $z\in\R$, $t\geq 0$ and satisfies $\widetilde{P}_k(z,0)=I$.

As definition \eqref{eq:defP} is modified, we cannot directly use Theorem \ref{th:EuclidDecay} to get a decay estimate in the Euclidean norm. The idea of the proof of Theorem \ref{th:EuclidDecay} is to estimate each $P^m_n$-semi-norm decay separately, see \eqref{eq:Pinduction}. Combining them yields a decay estimate in the $P$-norm. We will now follow the idea of the proof of Theorem \ref{th:EuclidDecay} but have to carefully modify each step to work for $\wt{P}_k$.\\

First, one can easily verify that for $m=1,2$ the estimate \eqref{eq:Pinduction} remains true, if we replace $P_k(z,0)$ by $\wt{P}_k(z,0)=I$:
\begin{align}\label{eq:Pk1}
|y_k(z,t)|^2_{P^1_{1,k}(z,0)} \leq e^{-2k^2 b(z)t} |y_k(z,0)|^2_2, \quad t\geq 0, k\neq 0,
\end{align}
and (using $c_2=6$)
\begin{align}\label{eq:Pk2}
|y_k(z,t)|^2_{P^2_{1,k}(z,0)} \leq \frac{6}{\min\{1,|\partial_z \lambda_k(z)|^2\}} (1 + k^4t^2)e^{-2k^2 b(z)t} |y_k(z,0)|^2_2,
\end{align}
for $t\geq 0$, $k\neq 0$.

Next we shall derive a similar estimate for the semi-norm $|\cdot|_{\wt{P}_{k}^3(z,0)}$.

%
%
\begin{lemma}\label{lem:CDsecond}
	Let $y_k(z,t)$ be a solution of the ODE \eqref{eq:secorderODE} and $z\in\R$, such that $\partial_z \lambda_k(z)\neq 0$. Then,
	\begin{align}
	|y_k(z,t) |^2_{\wt{P}_k^3(z,0)} \leq 146.25 \frac{1+|\partial_z^2 \lambda_k(z)|^2}{\min\{1,|\partial_z \lambda_k(z)|^4\}} (1+k^8t^4)e^{-2k^2 b(z)t} |y(z,0)|^2_2 \label{eq:FINALLY}
	\end{align}	
	for $t\geq 0$, $k\in\Z\setminus \{0\}$.
\end{lemma}
The technical proof is deferred to Appendix \ref{pr:CDsecond}.

Finally, we can estimate solutions to \eqref{eq:secorderODE} in Euclidean norm with the help of \eqref{eq:defPtilde}:
\begin{align*}
|y_k(z,t)|^2_2 &= |y_k(z,t)|_{\wt{P}_k(z,0)}^2\\
&=|y_k(z,t)|^2_{P^1_{1,k}(z,0)} + |\partial_z \lambda_k(z)|^2 |y_k(z,t)|^2_{P^2_{1,k}(z,0)}\\
&\qquad + 4|\partial_z \lambda_k(z)|^4 |y_k(z,t)|^2_{\wt{P}^3_{k}(z,0)}.
\end{align*}
Using \eqref{eq:Pk1} for $P^1_{1,k}(z,0)$ and \eqref{eq:Pk2} for $P_{1,k}^2(z,0)$ allows to estimate the first two terms. For the third term including $\wt{P}^3_k(z,0)$, we use \eqref{eq:FINALLY} to get
\begin{align*}
|y_k(z,t)|^2_2 &\leq \Big[1+6 \max\{1,|\partial_z \lambda_k(z)|^2\}(1+k^4 t^2)\\
&\hspace{-12mm}+4\max\{1,|\partial_z \lambda_k(z)|^4\}( 1+ |\partial_z^2 \lambda_k(z)|^2)146.25(1+k^8t^4) \Big]e^{-2k^2 b(z) t} |y_k(z,0)|^2_2 \\
\quad&\begin{aligned} \leq\Big[1+\big(12+585 ( 1+ |\partial_z^2 \lambda_k(z)|^2)\big)\max\{1,|\partial_z \lambda_k(z)|^4\}\Big] \\
\times(1+k^8t^4)e^{-2k^2 b(z) t} |y_k(z,0)|^2_2\end{aligned}\numberthis \label{eq:sens2case3}
\end{align*}
for $t\geq 0$, $k\in\Z\setminus\{0\}$.
Most notably, as $\partial_z a,\partial_z b,\partial_z^2 a,\partial_z^2b\in L^\infty(\R)$, the multiplicative constant
\begin{align*}
\wt{\mathscr{C}}_k(z) := 1+\big(12+585 ( 1+ |\partial_z^2 \lambda_k(z)|^2)\big)\max\{1,|\partial_z \lambda_k(z)|^4\}
\end{align*}
is uniformly bounded in $z\in\R$. This includes the problematic limit $\partial_z \lambda_k(z) \to 0$ in combination with $\partial^2_z\lambda_k(z)\not= 0$ (defect 2 to defect 1), which is our desired result for Case 3.\\

Combining Cases 1--3 for $z\in\R$ leads to:
\begin{theorem}
Let $a,b\in C^2(\R)$ where $b_0:=\inf_{z\in\R}b(z)>0$ and \linebreak$b,\partial_z a,\partial_z b,\partial_z^2 a, \partial_z^2 b\in L^\infty(\R)$. Then, there exists a constant $\mathscr{C}>0$, such that normalized  solutions $y(x,z,t)=(u,v,w)^T$ to the system of equations \eqref{eq:heat}, \eqref{eq:sensheat} and \eqref{eq:sens2heat} with steady state $y^\infty := (1,0,0)^T$ satisfy
\begin{align*}
\sup_{z\in\R}\|y(\cdot,z,t)-y^{\infty}\|^2_{L^2(0,2\pi;\R^3)} \leq \mathscr{C}(1+t^4)e^{-2b_0 t} \sup_{z\in\R}\|y(\cdot,z,0)-y^\infty \|^2_{L^2(0,2\pi;\R^3)}
\end{align*}
for $t\geq 0$.
\end{theorem}
\begin{proof}
Analogous to the first order sensitivity equations, combining the three above cases of defects of $D_k(z)$, leads to a decay estimate uniform in $z\in\R$. Due to the estimates \eqref{eq:sens2case1}, \eqref{eq:sens2case2} and \eqref{eq:sens2case3}, there exists an $\wt{\mathscr{C}}>0$ independent of $z\in\R$ and $k\in\Z\setminus\{0\}$ such that
\begin{align*}
|y_k(z,t)|^2_2 \leq \wt{\mathscr{C}} (1+k^8t^4 ) e^{-k^2 b(z) t} |y_k(z,0)|^2_2, \quad t\geq 0.
\end{align*}
With Parseval's identity (in analogy to the proof of Theorem \ref{th:sens1Heat}) the desired result follows.\qed
\end{proof}

\subsection{Decay estimates with Duhamel's formula}\label{subsec:duhamel}
Another method to get decay estimates with sharp rate for sensitivity equations is to use Duhamel's formula instead of the above presented Lyapunov functional method.

In the case of the Fourier transformed linear heat-convection equation \eqref{eq:heatFourier}, the solution  $u_k(z,t)$ is given explicitly as $u_k(z,t) = e^{-k^2 \lambda_k(z) t} u_k(z,0)$. We can interpret the Fourier transformed first order sensitivity equation \eqref{eq:sensheatFourier} as an inhomogeneous equation of form
\begin{align*}
\partial_t v_k(z,t)  +k^2 \lambda_k(z) v_k(z,t) = g_k(z,t)
\end{align*}
with $g_k(z,t):=-k^2 (\partial_z\lambda_k(z)) u_k(z,t)$.
By Duhamel's formula we get
\begin{align*}
v_k(z,t) &= e^{-k^2 \lambda_k(z) t}v_k(z,0) - k^2(\partial_z\lambda_k(z)) \int_0^t e^{-k^2 \lambda_k(z)(t-s)} e^{-k^2 \lambda_k(z) s}u_k(z,0) ds\\
&=e^{-k^2 \lambda_k(z) t}v_k(z,0) - k^2(\partial_z\lambda_k(z))t e^{-k^2 \lambda_k(z) t} u_k(z,0).
\end{align*}
By using the solution propagator norm \eqref{eq:evpropest} with $\epsilon t = k^2t$, this yields the decay estimate for each Fourier-mode $y_k(z,t)=(u_k,v_k)^T$, $k\in\Z\setminus\{0\}$:
\begin{align*}
|y_k(z,t)|^2_2 \leq  \frac{4}{3} (1+k^4t^2 ) e^{-2k^2 b(z) t} |y_k(z,0)|^2_2,\quad t\geq 0.
\end{align*}
By iteration, Duhamel's formula gives a decay estimate for sensitivity equations of \emph{arbitrary} order. A similar method of iteratively deducing decay estimates was used e.g.\ in \cite{JLM} (see Theorems 4.1 and 4.2) and \cite{LiuJin} (see Theorems 2.1 and 4.4).

%


\section{Two velocity BGK model with uncertain coefficients}\label{sec:bgk}

Our starting point is the linear one-dimensional BGK-model
for the probability density $f(x,v,t)\geq 0$. This kinetic equation reads
\begin{align}\label{eq:bgk}
\partial_t f +v \partial_x f
= M_{T}(v) \int_\R f(x,v,t) dv - f(x,v,t),
\end{align}
for $x\in\T^1$, velocities $v\in\R$, and the Maxwell distribution \tw{$M_{T}(v)=(2\pi T)^{-\frac{1}{2}} e^{-\frac{|v|^2}{2T}}$}, with given temperature $T$. Exponential decay towards the equilibrium for this $v$-continuous model was proved in \S 4.3 of \cite{Achleitner2016}.
We reduce the model drastically and allow only for two discrete velocities $v_\pm = \pm 1$, denoting $f_\pm(x,t):=f(x,\pm 1,t)$. This leads to the system of equations
\begin{align*}
\partial_t f_+(x,t) &= - \partial_x f_+(x,t) + \sigma(f_-(x,t)-f_+(x,t)),\\
\partial_t f_-(x,t) &= \partial_x f_-(x,t) - \sigma(f_-(x,t)-f_+(x,t)),
\end{align*}
called \emph{Goldstein-Taylor model} with the relaxation coefficient $\sigma>0$. These equations serve as a toy model that still exhibits many features of \eqref{eq:bgk}. For $\sigma=\frac12$, an explicit exponential decay rate of the two velocity model by means of Lyapunov functionals was shown in \S 1.4 of \cite{Dolbeault2015}. The sharp decay estimate was found in \cite{Achleitner2016}, \S 4.1 with a refined functional. We are interested here in augmenting the large-time analysis with a sensitivity analysis.

\subsection{First order parameter sensitivity analysis}\label{subsec:bgkintro}
Similarly to \S \ref{sec:hc} we allow the relaxation coefficient to contain uncertainty and denote it by $\sigma(z)$. Throughout \S \ref{sec:bgk}, assume $\sigma\in C^1(\R)$, $\partial_z\sigma\in L^\infty(\R)$, $\sigma_0:=\inf_{z\in\R}\sigma(z)>0$, and $\sigma_1:=\sup_{z\in\R}\sigma(z)<2$. This leads to the following equations for $x\in \T^1$, the parameter $z\in\R$ and $t\geq 0$:
\begin{align}\label{eq:2veloBGK}
\begin{aligned}
\partial_t f_+(x,z,t) &= -\partial_x f_+(x,z,t) + \frac{\sigma(z)}{2}(f_-(x,z,t) - f_+(x,z,t)),\\
\partial_t f_-(x,z,t) &= \partial_x f_-(x,z,t) - \frac{\sigma(z)}{2}(f_-(x,z,t)-f_+(x,z,t)),
\end{aligned}
\end{align}
with initial condition
\begin{align*}
f_\pm(x,z,0)=f_\pm^0(x,z).
\end{align*}
For each $z\in\R$ assume $f^0_\pm(\cdot,z,t)\in L^1_+(\T^1)$.

The model is conserving total mass (in time), i.e. $\int_0^{2\pi} [f_+(x,z,t)\linebreak
+f_-(x,z,t)] dx = \emph{const}.$ for all $z\in\R$. The unique normalized steady state for the system is given as $f^\infty_{+}(z)=f_{-}^{\infty}(z)=\frac{1}{2}$. Correspondingly, we shall also assume that the initial total mass is normalized, as\ $\frac{1}{2\pi}\int^{2\pi}_0 [f_+^0(x,z) + f_-^0(x,z)] dx = 1$.

To analyze the (linear order) sensitivity of the equation with respect to the relaxation function $\sigma(z)$, we investigate the corresponding family of sensitivity equations for $g_\pm(x,z,t):=\partial_z f_\pm(x,z,t)\in\R$. For each $z\in\R$, they are given as
\begin{align}\label{eq:2veloBGKsen}
\begin{aligned}
\partial_t g_+(x,z,t) &= - \partial_x g_+(x,z,t)+ \frac{\sigma(z)}{2}(g_-(x,z,t)-g_+(x,z,t)) \\
&\qquad+ \frac{\partial_z\sigma(z)}{2}(f_-(x,z,t)-f_+(x,z,t)),\\
\partial_t g_-(x,z,t) &= \partial_x g_-(x,z,t)- \frac{\sigma(z)}{2}(g_-(x,z,t)-g_+(x,z,t))\\
&\qquad - \frac{\partial_z\sigma(z)}{2}(f_-(x,z,t)-f_+(x,z,t)).
\end{aligned}
\end{align}
For each $z\in\R$, this system \eqref{eq:2veloBGKsen} is also conserving total mass (in time), i.e.\ $\int_0^{2\pi}[g_+(x,z,t)+g_-(x,z,t)]dx=const.$ Due to the normalization of $f_\pm^0(x,z)$, we have $\frac{1}{2\pi}\int_0^{2\pi}[g^0_+(x,z)+g^0_-(x,z)]dx=0$ with the corresponding steady state given as $g_+^\infty(x,z)=g_-^\infty(x,z)=\partial_z f_\pm^\infty(x,z) = 0$.

To analyze the decay behavior of solutions of the above system \eqref{eq:2veloBGK}--\eqref{eq:2veloBGKsen}, we consider the Fourier series $f_\pm(x,z,t)=\sum_{k\in\Z} f_{\pm,k}(z,t) e^{ikx}$. It is convenient to introduce the following linear combinations of the Fourier modes $f_{\pm,k},\, k\in\Z$:
\begin{align*}
u_{k}(z,t):=\begin{pmatrix}
f_{+,k}(z,t)+f_{-,k}(z,t)\\
f_{+,k}(z,t)-f_{-,k}(z,t)
\end{pmatrix}.
\end{align*}
For each $k\in\Z$ they satisfy the decoupled ODE system
\begin{align}\label{eq:ODEAk}
\partial_t u_k(z,t) =
-\underbrace{\begin{pmatrix}
0~&ik\\
ik~& \sigma(z)
\end{pmatrix}}_{\mathllap{A_k(z):=}}u_k(z,t).
\end{align}
For $k\in\Z$, the matrix $A_k(z)$ has the eigenvalues $$\lambda_{\pm,k}(z): = \frac{\sigma(z)}{2}\pm i \sqrt{k^2-\frac{\sigma^2(z)}{4}}.$$ Note that the discriminant is always positive for $k\neq 0$, due to our assumption $0<\sigma(z)<2$.

The eigenvectors are given by
\begin{align*}
\hat{v}_{\pm,k}(z): =
\begin{pmatrix}
\frac{i\lambda_\mp(z)}{k}, &
1
\end{pmatrix}^T\hspace{-2mm},\, k\in\Z\setminus\{0\},
\qquad \text{and} \qquad
\hat{v}_{+,0}:=\begin{pmatrix} 1\\ 0 \end{pmatrix},\, \hat{v}_{-,0}:=\begin{pmatrix} 0\\ 1 \end{pmatrix}.
\end{align*}
Similarly to \eqref{eq:ODEAk}, the Fourier modes
\begin{align*}
w_{k}(z,t):=\begin{pmatrix}
g_{+,k}(z,t)+g_{-,k}(z,t)\\
 g_{+,k}(z,t)-g_{-,k}(z,t)
\end{pmatrix},
\end{align*}
with $g_{\pm,k}=\partial_z f_{\pm, k}$, satisfy the ODE systems
\begin{align*}
\partial_t w_k(z,t) =
-\begin{pmatrix}
0~&ik\\
ik~&\sigma(z)
\end{pmatrix}w_k(z,t) -
\begin{pmatrix}
0~ & 0\\
0~ & \partial_z\sigma(z)
\end{pmatrix}u_k(z,t),\quad k\in\Z.
\end{align*}
Combining the two ODE systems for $u_k$ and $w_k$ leads to the following $4\times 4$-systems describing the first order sensitivity equations for the model \eqref{eq:2veloBGK} in Fourier space:
\begin{align}\label{eq:4times4}
\partial_t
\underbrace{\begin{pmatrix}
u_k(z,t)\\ w_k(z,t)
\end{pmatrix}}_{\udef{y_k(z,t):=}}
=
-\underbrace{\left(\begin{array}{cc|cc}
0~ & ik & ~~0~~& 0\\
ik &\sigma(z)  & 0&0\\
\hline
0&0&0&ik\\
0&\partial_z\sigma(z)&ik&\sigma(z)
\end{array}
\right)}_{\mathllap{D_k(z):=}}
\begin{pmatrix}
u_k(z,t)\\w_k(z,t)
\end{pmatrix},\quad k\in\Z.
\end{align}
Due to the block triangular form of the matrix $D_k(z)$, the eigenvalues of $D_k(z)$ are not affected by $\partial_z\sigma(z)$:
\begin{align*}
\lambda_{\pm,k}(z): = \frac{\sigma(z)}{2}\pm i \sqrt{k^2-\frac{\sigma^2(z)}{4}},\quad k\in\Z,
\end{align*}
where both eigenvalues have algebraic multiplicity two.

For $k\in\Z$ the matrix $D_k(z)$ is defective, if and only if $\partial_z\sigma(z) \neq 0$: For $k=0$ only the eigenvalue $\lambda_{+,0}(z)=\sigma(z)$ is defective of order one, and for $k\neq 0$ both eigenvalues $\lambda_{\pm,k}(z)$ are defective of order one.
\subsection{Sharp decay estimates for the parameter sensitivity equations}
The reasons for the assumptions from \S \ref{subsec:bgkintro} imposed on $\sigma(z)$ will become evident in the following analysis: The lower bound $\sigma_0>0$ is needed in order to get a uniform in $z\in\R$ decay rate. The assumptions $\sigma_1 <2$ and $\|\partial_z \sigma\|_\infty<\infty$ are necessary for the multiplicative constant in the decay estimate (obtained by Theorem \ref{th:EuclidDecay}) to be bounded for all $z\in\R$.

The decay rate of each mode $k\in\Z$ is determined by the size of the spectral gap of $D_k(z)$, which we denote by $\mu_k(z)>0$, and its defectiveness. For $D_0(z)$ the eigenvalues are $\lambda_{+,0}(z)=\sigma(z)$ and $\lambda_{-,0}(z)=0$. Hence, the spectral gap for the zeroth mode is $\mu_0(z)=\sigma(z)$. The zeroth mode of the steady state in our transformed setting is given as
$y_0^\infty = (f_{+,0}^\infty + f_{-,0}^\infty,f_{+,0}^\infty - f_{-,0}^\infty, g_{+,0}^\infty + g_{-,0}^\infty, g_{+,0}^\infty - g_{-,0}^\infty)^T= (1, 0,0,0)^T$.
This implies that any solution to $\eqref{eq:4times4}$ for $k=0$, $z\in\R$ fulfills the decay estimate
\begin{align}\label{eq:case0}
\left|y_0(z,t)-y_0^\infty\right|^2_2 \leq \mathscr{C}_0(z) (1+t^2) e^{-2 \sigma(z) t}\left|y_0(z,0)-y_0^\infty\right|^2_2,\quad t\geq 0,
\end{align}
with the constant
\begin{align*}
\mathscr{C}_0(z) = 12\cdot \max\{2, 1+|\partial_z\sigma(z)|^2\}\leq 12\cdot \max\{2,1+\|\partial_z \sigma\|_\infty^2\}
\end{align*}
that can be computed in analogy to Case 2 in \S \ref{subsec:hcSens1}. Note that Theorem \ref{th:EuclidDecay} can only be applied here to the two-dimensional subspace of $\C^4$ that pertains to $\lambda_{+,0}(z)=\sigma(z)$. In the orthogonal subspace corresponding to $\lambda_{-,0}=0$, we have $y_{-,0}(z,t)=y_{-,0}(z,0) = y_{-,0}^\infty = (1, *, 0, *)^T$, where `$*$' denotes the elements of $y_{+,0}$.

For the modes $k\in\Z\setminus\{0\}$ the matrix $D_k(z)$ is positive stable and the spectral gap is independent of $k$ (in contrast to the examples in \S \ref{sec:hc} and \S \ref{sec:fpe}):
\begin{align*}
\mu_k(z):=\min\{\re (\lambda_{+,k}(z)),\re (\lambda_{-,k}(z))\}= \frac{\sigma(z)}{2}.
\end{align*}
Moreover, the steady state is given as $y_k^\infty = 0\in\C^4$.\bigskip

In the next step, we apply Theorem \ref{th:EuclidDecay} to the system \eqref{eq:4times4} to get a sharp decay estimate for each Fourier mode $k\in\Z\setminus\{0\}$ of type
\begin{align*}
|y_k(z,t)-y_k^\infty |^2_2 \leq \mathscr{C}_k(z) 2(1+t^2) e^{-\sigma(z)t} |y_k(z,0)-y_k^\infty|^2_2.
\end{align*}
A summation over the estimates for all Fourier modes will allow us to apply Parseval's identity on the left-hand side. In order to apply it also on the right-hand side one requires a uniform in $k$ and $z$ bound of the multiplicative decay constant $\mathscr{C}_k(z)$. We shall derive this bound now.

For each $k\in\Z\setminus\{0\}$, the matrix $P_k(z,0)$ of Theorem \ref{th:EuclidDecay} has to be chosen depending on the defectiveness of the matrix $D_k(z)$, which is determined by the value of $\partial_z\sigma(z)$.

\subsubsection*{Case 1; $z\in\R$ such that $\partial_z\sigma(z)=0$:}
The matrix $D_k(z)$ is non-defective, and we construct the matrix $P_k(z,0)$ according to \eqref{eq:defP}: $N=4$, with  $l_n=1$ for $n\in\{1,\ldots,4\}$, $M=1$, i.e. each $n$ is in Case 1 of \S \ref{sec:ODE}. Choosing $\beta_{n,k}=1$ leads to
\begin{align*}
P_k(z,0):=v_{1,+,k}^{(0)}\otimes v_{1,+,k}^{(0)}+v_{1,-,k}^{(0)}\otimes v_{1,-,k}^{(0)}+v_{2,+,k}^{(0)}\otimes v_{2,+,k}^{(0)}+v_{2,-,k}^{(0)}\otimes v_{2,-,k}^{(0)},
\end{align*}
and the eigenvectors of $D^H_k(z)$ corresponding to $\overline{\lambda}_{\mp,k}(z)=\lambda_{\pm,k}(z)$ (i.e.\ satisfying the equation $D^H_k v^{(0)}_{i,\pm,k} = \lambda_{\pm,k}v^{(0)}_{i,\pm,k}$ for $i=1,2$) are given as
\begin{align*}
&v_{1,\pm,k}^{(0)}(z) =
\begin{pmatrix}
-\frac{i\lambda_{\mp,k}(z)}{k},& 1, & 0,& 0
\end{pmatrix}^T,
&v_{2,\pm,k}^{(0)}(z) =
\begin{pmatrix}
 0,& 0, & -\frac{i\lambda_{\mp,k}(z)}{k},& 1
\end{pmatrix}^T.
\end{align*}
For each fixed value $\sigma\in [\sigma_0,\sigma_1]$, we have
\begin{align*}
\lim_{k\to+\infty}\tilde{v}_{1,\pm,k}^{(0)}(\sigma) =
\begin{pmatrix}
\mp1\\ 1\\ 0\\ 0
\end{pmatrix},&&
\lim_{k\to+\infty} \tilde{v}_{2,\pm,k}^{(0)}(\sigma) =
\begin{pmatrix}
0\\ 0\\ \mp1\\ 1
\end{pmatrix},
\end{align*}
as well as
\begin{align*}
\lim_{k\to-\infty} \tilde{v}_{1,\pm,k}^{(0)}(\sigma) =
\begin{pmatrix}
\pm1\\ 1\\ 0\\ 0
\end{pmatrix},&&
\lim_{k\to-\infty} \tilde{v}_{2,\pm,k}^{(0)}(\sigma) =
\begin{pmatrix}
0\\ 0\\ \pm1\\ 1
\end{pmatrix},
\end{align*}
where we used the notations $\tilde{v}_{i,\pm,k}^{(0)}(\sigma(z)) = v_{i,\pm,k}^{(0)}(z)$ for $i=1,2$. Denoting $\tilde{P}_k(\sigma(z),0)=P_k(z,0)$,
it follows that
\begin{align}\label{eq:PkconvC1}
\lim_{|k|\to \infty} \max_{\sigma\in[\sigma_0,\sigma_1]}| \tilde{P}_k(\sigma,0) -2 I|_2=0.
\end{align}
For each $\sigma\in[\sigma_0,\sigma_1]$ and $k\neq 0$, the four vectors $\tilde{v}_{1,\pm,k}(\sigma)$ and $\tilde{v}_{2,\pm,k}(\sigma)$ are linearly independent and hence, the matrix $\tilde{P}_k(\sigma,0)$ is positive definite. As all entries of $\tilde{P}_k(\sigma,0)$ are continuous in $\sigma\in[\sigma_0,\sigma_1]$ and  the eigenvalues are continuous with respect to the matrix entries, we get
\begin{align*}
\inf_{z\in\R} \lambda_{\min}^{P_k(z,0)} \geq \min_{\sigma\in[\sigma_0,\sigma_1]} \lambda^{\tilde{P}_k(\sigma,0)}_{\min} =: \lambda_{k,\min}>0.
\end{align*}
Similarly, we get
\begin{align*}
\sup_{z\in\R} \lambda_{\max}^{P_k(z,0)} \leq \max_{\sigma\in[\sigma_0,\sigma_1]} \lambda_{\max}^{\tilde{P}_k(\sigma,0)}=:\lambda_{k,\max} <\infty.
\end{align*}
Because of \eqref{eq:PkconvC1} we have $\lambda_{k,\max},\lambda_{k,\min}\to 2$ for $|k|\to\infty$, and therefore \begin{align*}
\lambda_{\min}:=\min_{k\in\Z\setminus\{0\}} \lambda_{k,\min} >0,&& \lambda_{\max}:=\max_{k\in\Z\setminus\{0\}} \lambda_{k,\max} < \infty.
\end{align*}
We summarize Case 1: For all $z\in\R$ such that $D_k(z)$ is non-defective, Theorem \ref{th:EuclidDecay} yields the decay estimate
\begin{align}\label{eq:modedecayNonDef}
\left|y_k(z,t)-y_k^\infty\right|^2_2 \leq 2\mathscr{C}_k(z)e^{-\sigma(z) t}|y_k(z,0)-y_k^\infty|^2_2,
\end{align}
with a uniform bound for the constants $\mathscr{C}_k(z)$ (defined in \eqref{eq:constantM}), i.e.
\begin{align*}
0<\mathscr{C}_k(z) = (\lambda_{\min}^{P_k(z,0)})^{-1}\lambda_{\max}^{P_k(z,0)} \leq (\lambda_{\min})^{-1}\lambda_{\max} =:\mathscr{C}<\infty,
\end{align*}
for $z\in\R$ and $k\neq 0$.

\subsubsection*{Case 2; $z\in\R$ such that $\partial_z\sigma(z)\neq 0$:} The two eigenvalues $\lambda_{\pm,k}(z)$ of $D_k(z)$ are both defective.
The eigenvectors and generalized eigenvectors of $D^H_k(z)$ corresponding to $\overline{\lambda}_{\mp,k}(z)=\lambda_{\pm,k}(z)$ (i.e.\ the generalized eigenvectors satisfy $D^H_k v_{\pm,k}^{(1)}=\lambda_{\pm,k}v_{\pm,k}^{(1)} + v^{(0)}_{\pm,k}$) are given as
\begin{align*}
&v_{\pm,k}^{(0)}(z) =
\begin{pmatrix}
-\frac{i\lambda_{\mp,k}(z)}{k},& 1, & 0,& 0
\end{pmatrix}^T,
\\
&v_{\pm,k}^{(1)}(z) =
\begin{pmatrix}
\frac{i\lambda^2_{\mp,k}(z)}{2k^3},&
\frac{\lambda_{\mp,k}(z)}{2k^2},&
\frac{-i\lambda_{\mp,k}(z)}{\sigma_z(z)k}(1-\frac{\lambda^2_{\mp,k}(z)}{k^2}),&
\frac{1}{\sigma_z(z)}(1-\frac{\lambda^2_{\mp,k}(z)}{k^2})
\end{pmatrix}^T.
\end{align*}
Now we construct the matrix $P_k(z,0)$ according to \eqref{eq:defP}: $N=2$, with  $l_+=l_-=M=2$, i.e. both $n\in\{+,-\}$ are in Case 3 of \S \ref{sec:ODE}. Choosing $\beta^1_{\pm,k}=1$ and $\beta^2_{\pm,k}=\frac{ (\sigma_z(z))^2}{4}$, leads to
\begin{align*}
P_k(z,0)&:= v_{+,k}^{(0)}\otimes v_{+,k}^{(0)}+\frac{ (\sigma_z(z))^2}{4}v_{+,k}^{(1)}\otimes v_{+,k}^{(1)}\\
&\qquad \qquad +v_{-,k}^{(0)}\otimes v_{-,k}^{(0)}+\frac{ (\sigma_z(z))^2}{4}v_{-,k}^{(1)}\otimes v_{-,k}^{(1)}.
\end{align*}
As mentioned in \S \ref{sec:hc} the specific choice of weights $\beta_\pm^m$ is crucial in order to get a uniformly bounded constant $\mathscr{C}_k(z)$  in the \emph{non-defective limit} $\partial_z \sigma(z)\to 0$.

Abbreviating $L:=\|\partial_z\sigma\|_\infty$, for each value $(\sigma,\sigma_z)\in [\sigma_0,\sigma_1]\times[-L,L]$ (with notation in analogy to Case 1), we have
\begin{align*}
\lim_{k\to+\infty}\tilde{v}_{\pm,k}^{(0)}(\sigma,\sigma_z) =
\begin{pmatrix}
\mp1\\ 1\\ 0\\ 0
\end{pmatrix},&&
\lim_{k\to+\infty}\frac{\sigma_z}{2} \tilde{v}_{\pm,k}^{(1)}(\sigma,\sigma_z) =
\begin{pmatrix}
0\\ 0\\ \mp1\\ 1
\end{pmatrix},
\\
\end{align*}
and
\begin{align*}
\lim_{k\to-\infty} \tilde{v}_{\pm,k}^{(0)}(\sigma,\sigma_z) =
\begin{pmatrix}
\pm1\\ 1\\ 0\\ 0
\end{pmatrix},&&
\lim_{k\to-\infty}\frac{\sigma_z}{2} \tilde{v}_{\pm,k}^{(1)}(\sigma,\sigma_z) =
\begin{pmatrix}
0\\ 0\\ \pm 1\\ 1
\end{pmatrix}.
\end{align*}
It follows that
\begin{align}\label{eq:defectivePk}
\lim_{|k|\to \infty}  \max_{(\sigma,\sigma_z)\in[\sigma_0,\sigma_1]\times[-L,L]} |\tilde{P}_k(\sigma,\sigma_z,0) -  2I|_2= 0.
\end{align}
For each $(\sigma,\sigma_z)\in[\sigma_0,\sigma_1]\times[-L,L]$, the four vectors $\tilde{v}^{(0)}_{\pm,k}(\sigma,\sigma_z)$ and $\sigma_z \tilde{v}^{(1)}_{\pm,k}(\sigma,\sigma_z)$ are linearly independent. This can be checked by considering the last two components of $\sigma_z \tilde{v}^{(1)}_{\pm,k}(\sigma,\sigma_z)$: They have the same form as the last two components of $\tilde{v}^{(0)}_{2,\pm,k}(\sigma,\sigma_z)$ from Case 1 above, up to a multiplicative factor that is non-zero for all $k\neq 0$ due to $\sigma_1<2$.

Hence, the matrix $\tilde{P}_k(\sigma,\sigma_z,0)$ is positive definite on $[\sigma_0,\sigma_1]\times[-L,L]$. Due to the specific choice of $\beta^2_{\pm,k}=\frac{ (\sigma_z(z))^2}{4}$ all  entries of $\tilde{P}_k(\sigma,\sigma_z,0)$ are continuous with respect to $(\sigma,\sigma_z)$. With the same argument as in Case 1, we get
\begin{align*}
 \inf_{z\in\R} \lambda_{\min}^{P_k(z,0)} \geq \min_{(\sigma,\sigma_z)\in[\sigma_0,\sigma_1]\times[-L,L]} \lambda^{\tilde{P}_k(\sigma,\sigma_z,0)}_{\min} =:\lambda_{k,\min}>0,
\end{align*}
and
\begin{align*}
\sup_{z\in\R} \lambda_{\max}^{P_k(z,0)} \leq \max_{(\sigma,\sigma_z)\in[\sigma_0,\sigma_1]\times[-L,L]} \lambda_{\max}^{\tilde{P}_k(\sigma,\sigma_z,0)}:=\lambda_{k,\max} <\infty.
\end{align*}
The limit \eqref{eq:defectivePk} implies
\begin{align*}
\lambda_{\min}:=\min_{k\in\Z\setminus\{0\}} \lambda_{k,\min} >0,&& \lambda_{\max}:=\max_{k\in\Z\setminus\{0\}} \lambda_{k,\max} < \infty.
\end{align*}
We summarize Case 2: For all $z\in\R$ such that $D_k(z)$ is defective, Theorem \ref{th:EuclidDecay} yields the decay estimate
\begin{align}\label{eq:modedecayDef}
\left|y_k(z,t)-y_k^\infty\right|^2_2 \leq \mathscr{C}_k(z) (1+t^2) e^{-\sigma(z) t}|y_k(z,0)-y_k^\infty|^2_2.
\end{align}
Here, the constants $\mathscr{C}_k(z)$ from \eqref{eq:constantM}, using $c_2 = 6$, are uniformly bounded since $\partial_z \sigma \in L^\infty(\R)$:
\begin{align*}
0<\mathscr{C}_k(z) &= 12 \cdot (\lambda_{\min}^{P_k(z,0)})^{-1}\lambda_{\max}^{P_k(z,0)} \Big(1+\frac{\frac{ (\sigma_z(z))^2}{4}}{\min\Big\{1,\frac{ (\sigma_z(z))^2}{4}\Big\}}\Big) \\
&\leq 12\cdot(\lambda_{\min})^{-1}\lambda_{\max} \max\{2,1+\frac{\|\partial_z\sigma\|_\infty^2}{4}\}=:\mathscr{C}<\infty
\end{align*}
for all $z\in\R$ and $k\neq 0$.\bigskip

Now, we have all the necessary ingredients to estimate the decay of solutions to the system \eqref{eq:2veloBGK}--\eqref{eq:2veloBGKsen}. Denote
\begin{align*}
&\Phi:=\begin{pmatrix}
f_+,&
f_-,&
g_+,&
g_-
\end{pmatrix}^T, &&\Phi^\infty := \begin{pmatrix}
\frac12,& \frac12,&0&0
\end{pmatrix}^T,\\
& y:=\begin{pmatrix}
f_+ + f_-,& f_+-f_-,& g_++g_-,& g_+-g_-
\end{pmatrix}^T, &&y^\infty:=
\begin{pmatrix}
1, & 0, & 0, & 0
\end{pmatrix}^T.
\end{align*}

\begin{theorem}
Let $\sigma\in C^1(\R)$ where $\sigma_0:=\inf_{z\in\R}\sigma(z)>0$, $\sigma_1:=\sup_{z\in\R}\sigma(z)<2$ and $\partial_z \sigma\in L^\infty(\R)$. Then, there exists a constant $\mathscr{C}>0$, such that normalized solutions $\Phi(x,z,t)$ of the system \eqref{eq:2veloBGK}--\eqref{eq:2veloBGKsen} satisfy
\begin{align*}
\sup_{z\in\R}\|\Phi(\cdot,z,t)-\Phi^{\infty}\|^2_{L^2(0,2\pi;\R^4)} \leq \mathscr{C}(1+t^2)e^{-\sigma_0 t} \sup_{z\in\R}\|\Phi(\cdot,z,0)-\Phi^\infty \|^2_{L^2(0,2\pi;\R^4)}
\end{align*}
for $t\geq 0$.
\end{theorem}
\begin{proof}
Applying Parseval's identity and the decay estimates \eqref{eq:case0}, \eqref{eq:modedecayNonDef} and \eqref{eq:modedecayDef}, where the constants $\mathscr{C}_k(z)$ are uniformly bounded in $z\in\R$, $k\in\Z$, one obtains
\begin{align*}
\|
\Phi(\cdot,z,t)-
\Phi^\infty \|^2_{L^2(0,2\pi;\R^4)}&= \frac12
\|y(\cdot,z,t)-y^\infty\|^2_{L^2(0,2\pi;\R^4)} \\
&= \frac{1}{4\pi}\sum_{k\in\Z} |y_k(z,t)-y_k^\infty |^2_2\\
&\leq \frac{1}{4\pi}\sum_{k\in\Z} 2\mathscr{C}_k(z) (1+t^2) e^{-\sigma(z) t} |y_k(z,0) - y_k^\infty |^2_2\\
&\leq
\mathscr{C} (1+t^2)e^{-\sigma_0 t} \|\Phi(\cdot,z,0)-\Phi^\infty \|^2_{L^2(0,2\pi;\R^4)}
\end{align*}
for all $t\geq 0$.	\qed
\end{proof}
\section{Fokker--Planck equations with uncertain coefficients}\label{sec:fpe}
%
%
In this section we consider the Fokker--Planck equation (FPE) with the spatial variable $x\in\R$,
\begin{align}\label{eq:FPE}
\partial_t f(x,z,t) &= \partial_x[\partial_x f(x,z,t) + a(z) x f(x,z,t) ] = :L(z) f(x,z,t),\quad t\geq 0,
\end{align}
with initial condition
\begin{align*}
f(x,z,0)&= f^0(x,z).
\end{align*}
The drift parameter $a(z)$ depends only on the uncertainty parameter $z\in\R$ and we assume $a \in C^1(\R)$ with $a_0:=\inf_{z\in\R}a(z)>0$. As in \S \ref{sec:hc}--\S\ref{sec:bgk}, we want to analyze the sensitivity of the decay for solutions to the steady state w.r.t.\ $z\in\R$. Contrary to the previous examples, \emph{the steady state here also depends on $z$.}

For each $z\in\R$ the unique normalized steady state of the equation, i.e. $L(z)f^\infty(x,z)=0$ with $\int_\R f^\infty (x,z) dx = 1$,  is given as
\begin{align*}
f^\infty(x,z) = \sqrt{\frac{a(z)}{2\pi}}e^{-\frac{x^2}{2}a(z)}.
\end{align*}
Denoting $g(x,z,t):= \partial_z f(x,z,t)$, the first order linear sensitivity equation is given as
\begin{align}\label{eq:sensFPE}
\partial_t g(x,z,t) = L(z) g(x,z,t) + a_z(z) [x \partial_x f(x,z,t) +  f(x,z,t)],\quad x\in\R,t\geq 0,
\end{align}
for each fixed $z\in\R$. The corresponding steady state is given as $g^\infty(x,z):= \partial_z f^\infty(x,z)$ which satisfies $\int_\R g^\infty(x,z)dx = 0$.

A direct approach to estimate the decay of $g$ via Duhamel's formula (cf. \S \ref{subsec:duhamel}) is not (easily) feasible: While a decay estimates with sharp rate for $f(x,z,t)$ is available, the decay behavior of the term $x \partial_x f(x,z,t)$ is not immediate. In \cite{JinZhu}, the Duhamel approach was taken to obtain decay estimates for nonlinear Vlasov--Fokker--Planck equations, but those estimates were not sharp.

We choose to expand $f(x,z,t)$ and $g(x,z,t)$ into  eigenfunctions of $L(z)$. This allows us to use a recursive relation of the eigenfunctions $h_k(x,z)$ and $x \partial_x h_k(x,z)$ for $k\in\N_0$.

\subsection{Eigenfunctions of the FP-operator $L(z)$}
The normalized eigenfunctions of $L(z)$ on the weighted space $L^2((f^\infty)^{-1})$ (with the inner product $\langle f, g\rangle_{L^2((f^\infty)^{-1})}=\int_\R f g (f^\infty)^{-1}dx$) are rescaled Hermite functions. The \emph{probabilists' Hermite polynomials} are defined as
\begin{align*}
H_k(x) := (-1)^k e^{\frac{x^2}{2}} \frac{d^k}{dx^k}e^{-\frac{x^2}{2}},\quad k\in \N_0,x\in\R,
\end{align*}
and satisfy the recursion
\begin{align}\label{eq:HermiteRel}
H_k'(x) = x H_k(x) - H_{k+1}(x) = k H_{k-1}(x),\quad k\in\N, x\in\R.
\end{align}
The \emph{Hermite functions} are given as
\begin{align}\label{eq:defhermite}
\tilde{h}_k(x) := \frac{1}{\sqrt{2\pi k !}}H_k(x) e^{-\frac{x^2}{2}},\quad k\in\N_0,x\in\R,
\end{align}
satisfying (due to \eqref{eq:HermiteRel}) $\partial_x \tilde{h}_k(x) = -\sqrt{k+1}\, \tilde{h}_{k+1}(x)$. We further denote the \emph{Hermite functions rescaled by $a(z)$} as
\begin{align}\label{eq:DefReHermite}
h_k(x,z) := \sqrt{a(z)}\, \tilde{h}_k( x \sqrt{a(z)}).
\end{align}
This rescaling is chosen such that the Hermite functions $h_k(z,\cdot)$ are normalized in $L^2((f^\infty)^{-1})$ for all $k\in\N_0$. Notice that $f^\infty(x,z)=h_0(x,z)$ and $g^\infty(x,z) = -\frac{\partial_z a(z)}{\sqrt{2}a(z)}h_2(x,z)$.

For later use we note that, due to \eqref{eq:HermiteRel}, the rescaled Hermite functions satisfy
\begin{align}\label{eq:Hermitex}
xh_k(x,z) = \frac{1}{\sqrt{a(z)}}[ \sqrt{k+1}\, h_{k+1}(x,z) + \sqrt{k}\,h_{k-1}(x,z)],
\end{align}
for $k\in\N, x,z\in\R$. Using \eqref{eq:HermiteRel} again, this implies
\begin{align}\label{eq:HermitexDer}
x \partial_x h_k(x,z) = -\sqrt{k+1} [\sqrt{k+2} \,h_{k+2}(x,z) + \sqrt{k+1}\,h_k(x)],
\end{align}
for $k\in\N_0, x,z\in\R$.

The spectrum of $L(z)$, with $z\in\R$ fixed, is given as
\begin{align*}
\sigma(L(z)) = \{-a(z)k \mid k\in\N_0\}.
\end{align*}
An orthonormal basis of eigenfunctions for the FP-operator $L(z)$ on $L^2((f^\infty)^{-1})$ is given by the rescaled Hermite functions defined in \eqref{eq:DefReHermite} (see e.g.\ \cite{Risken}, \S 5.5.1, \S 10.1.4), i.e. for $z\in\R$ fixed:
\begin{align*}
L^2((f^\infty)^{-1}) = \bigoplus_{k\in\N_0}\linspan\{ h_k(\cdot,z)\},\qquad L(z)h_k(\cdot,z) = -a(z)k\, h_k(\cdot,z).
\end{align*}

\subsection{Sharp decay estimate for the parameter sensitivity equations}\label{subsec:FPEdecay}
Let us assume $f^0(\cdot,z)$, $g^0(\cdot,z)\in L^2((f^\infty)^{-1})$ where $f(x,z,t)$ is a probability density, $\int_\R f^0(x,z)dx = 1$, and $g^0(x,z)$ does not carry any mass, i.e. $0=\int_\R g^0(x,z)dx =(g^0(\cdot,z),h_0(\cdot,z))_{L^2((f^\infty)^{-1})}$. The eigenfunction expansions for the corresponding solutions $f(x,z,t)$ of \eqref{eq:FPE} and $g(x,z,t)$ of \eqref{eq:sensFPE} are given as
\begin{align*}
f(x,z,t) = \sum_{k=0}^\infty f_k(z,t) h_k(x,z),\qquad g(x,z,t) = \sum_{k=1}^\infty g_k(z,t) h_k(x,z),
\end{align*}
for $x,z\in\R$ and $t\geq 0$. Due to \eqref{eq:FPE}, each eigenmode evolves as
\begin{align}\label{eq:FPEfmodeeq}
\partial_t f_k(z,t) = -a(z) k f_k(z,t),\quad k\in\N_0,
\end{align}
and hence $f_0(z,t)=1$.
Plugging the eigenfunction expansion for $g$ into \eqref{eq:sensFPE} leads to
\begin{align*}
\sum_{k=1}^\infty \partial_t g_k(z,t) h_k(x)&= - a(z) \sum_{k=1}^\infty k g_k(z,t) h_k(x,z)\\
&\qquad \qquad+ a_z(z) \sum_{k=0}^\infty f_k(z,t)\left[ h_k(x,z)+x \partial_x h_k(x,z)\right].
\end{align*}
Applying identity \eqref{eq:HermitexDer} gives
\begin{align*}
\sum_{k=1}^\infty &\partial_t g_k(z,t) h_k(x)\\
& =- a(z) \sum_{k=1}^\infty k g_k(z,t) h_k(x,z)+ a_z(z) \sum_{k=0}^\infty f_k(z,t) h_k(x,z) \\
& \qquad +a_z(z) \sum_{k=0}^\infty- (k+1) f_k(z,t) h_k(x,z) - \sqrt{(k+1)(k+2)} f_k(z,t) h_{k+2}(x,z)\\
&=-a(z) \sum_{k=1}^\infty k g_k(z,t) h_k(x,z) - a_z(z) f_1(z,t) h_1(x,z)\\
&\qquad  - a_z(z) \sum_{k=2}^\infty [k f_k(z,t) + \sqrt{k(k-1)} f_{k-2}(z,t) ] h_k(x,z).
\end{align*}
Separating the eigenmodes then yields
\begin{align*}
&\partial_t g_1(z,t) = -a(z) g_1(z,t) -a_z(z) f_1(z,t),\\
&\partial_t g_k(z,t) = -ka(z)  g_k(z,t)-a_z(z)\big[k f_k(z,t) + \sqrt{k(k-1)} f_{k-2}(z,t)\big], \quad k\geq 2.\numberthis \label{eq:FPEsensmodeeq}
\end{align*}
In contrast to $f$, the $k$th modes of $g$ do not decouple for $k\geq 1$. They are rather coupled as the pair $(f_1,g_1)$, respectively the triples $(f_{k-2},f_k, g_k)$ for $k\geq 2$. For $k=1$ the evolution equation for $f_1(z,t)$, $g_1(z,t)$ can be written as the ODE system
\begin{align}
\partial_t \underbrace{\begin{pmatrix}
	f_1\\
	g_1
	\end{pmatrix}}_{\udef{y_1(z,t):=}}=
-a(z)\underbrace{\begin{pmatrix}
	1&~~0~~\\
	\alpha(z)&~~ 1~~
	\end{pmatrix}}_{\udef{C_1(z):=}}
\begin{pmatrix}
f_1\\
g_1
\end{pmatrix},
\end{align}
for $z\in\R$, $t\geq 0$, with the notation $$\alpha(z):=\frac{a_z(z)}{a(z)}.$$ For $k=2$, equation \eqref{eq:FPEsensmodeeq} can be written as
\begin{align*}
\partial_t \wt{g}_2(z,t) = -2 a(z)[\wt{g}_2(z,t)+ \alpha(z)f_2(z,t)],
\end{align*}
with $\wt{g}_2(z,t):= g_2(z,t) + \frac{\alpha(z)}{\sqrt{2}}$, since $f_0(z,t)= \int_\R f(x,z,t)dx \equiv 1$. The corresponding system of equations is given as
\begin{align}
\partial_t \underbrace{\begin{pmatrix}
	f_2\\
	\wt{g}_2
	\end{pmatrix}}_{\udef{y_2(z,t):=}}=
-2a(z)\underbrace{\begin{pmatrix}
	1&~~0~~\\
	\alpha(z)&~~ 1~~
	\end{pmatrix}}_{\udef{C_2(z):=}}
\begin{pmatrix}
f_2\\
\wt{g}_2
\end{pmatrix},
\end{align}
for $z\in\R$, $t\geq 0$. Since the matrices $C_1(z)=C_2(z)$ are defective, if and only if $a_z(z)\neq 0$, we shall now distinguish these cases.
\subsubsection*{Case $k=1,2$ and $z\in\R$ such that $a_z(z)=0$:} The matrices $C_1(z)=C_2(z)$ are diagonal and the solutions of the eigenmodes $k=1,2$ are given explicitly as
\begin{align*}
y_k(z,t) &= e^{-ka(z) t} y_k(z,0),\quad t\geq 0, k=1,2.
\end{align*}
The decay estimate
\begin{align}\label{eq:FPEmodek12diag}
|y_k(z,t)|_2^2 \leq e^{-2ka(z)t}|y_k(z,0)|^2_2, \quad t\geq 0, k=1,2,
\end{align}
follows.
\subsubsection*{Case $k=1,2$ and $z\in\R$ such that $a_z(z)\neq 0$:}
The matrices $C_1(z)=C_2(z)$ are defective of order 1 and we apply Theorem \ref{th:EuclidDecay} to get a uniform-in-$z$ decay estimate. The construction of the matrices $P_1(z,t)=P_2(z,t)$ resembles Example \ref{ex:Def1toDef0} (with $\epsilon = \alpha(z)$ and rescaling $t\mapsto ka(z) t$). It yields the decay estimate
\begin{align}\label{eq:FPEmodek12}
|y_k(z,t)|^2_2 \leq \mathscr{C}_k(z)(1+k^2a(z)^2t^2) e^{-2k a(z) t}|y_k(z,0)|^2_2,\quad k=1,2,
\end{align}
with the uniform in $z\in\R$ bounded constant (for $k=1,2$)
\begin{align*}
\mathscr{C}_k(z)=12\cdot \max\{2, 1+ \alpha(z)^2\} \leq 12 \max\left\{2,1+\frac{\|a_z\|^2_\infty}{a_0^2}\right\}=:\mathscr{C}_{1,2}.\numberthis \label{eq:FPEuniC1}
\end{align*}
Notice that by definition of $y_2(z,t)$ the decay $y_2(z,t)\overset{t\to\infty}{\longrightarrow} 0$ implies $g_2(z,t)\overset{t\to\infty}{\longrightarrow} -\frac{\alpha(z)}{\sqrt{2}}=(g^\infty(\cdot,z),h_2(\cdot,z))_{L^2((f^\infty)^{-1})}$. This concludes the analysis for the modes $k=0,1,2$.

For our goal to get a decay estimate with sharp uniform-in-$z$ decay rate of the system \eqref{eq:FPE}--\eqref{eq:sensFPE} as formulated in Theorem \ref{th:sensFPE} below, it is important to get a ``precise'' decay estimate for the modes $k=1,3$. Only these two modes have the spectral gap $a(z)$ of the system of equations \eqref{eq:FPE}--\eqref{eq:sensFPE}. The other modes have larger spectral gaps and decay much faster. On the level of the modal equations for $k\geq 4$ all we need are ``sufficient'' decay estimates, namely rates at least as good as the ones of the modes $k=1,3$. This is in contrast to \S \ref{sec:bgk} where every Fourier mode has the spectral gap $\frac{\sigma(z)}{2}$ of the sensitivity equations and needs ``precise'' treatment.
\bigskip

For the modes $k\geq 3$, the equation for $g_k(z,t)$ corresponds to
\begin{align}\label{eq:fpemodek3}
\partial_t \underbrace{\begin{pmatrix}
f_{k-2}\\
f_k\\
g_k
\end{pmatrix}}_{\udef{y_k(z,t):=}}=
-ka(z)\underbrace{
	\begin{pmatrix}
	\frac{k-2}{k}&~~0&~~~~~0~~\\
	0&~~1&~~~~~0~~\\
	\gamma(k)\alpha(z)& ~~\alpha(z)&~~~~~ 1~~
	\end{pmatrix}}_{\udef{C_k(z):=}}
\begin{pmatrix}
f_{k-2}\\
f_k\\
g_k
\end{pmatrix},
\end{align}
for $z\in\R$, $t\geq 0$, denoting $$\gamma(k):= \sqrt{\frac{k-1}{k}}\in [\textstyle\sqrt{\frac{2}{3}},1).$$
For each $k\geq 3$, the eigenvalues of $C_k(z)$ are $\lambda_{1,k}=\frac{k-2}{k}$ and $\lambda_{2,k}=1$, where $\lambda_{2,k}$ is defective of order 1, if and only if $a_z(z) \neq 0$.  The (non-defective) spectral gap of $C_k(z)$ is given as
$$\mu_k = \frac{k-2}{k}, \quad k\geq 3.$$
\subsubsection{Case $k= 3$ and $z\in\R$ such that $a_z(z) =0$:}
The matrix $C_3^H(z)$ is diagonal and the solutions for the eigenmodes are given explicitly as
\begin{align*}
f_1(z,t) &= e^{-a(z) t} f_1(z,0),\\
f_3(z,t) &= e^{-3a(z) t} f_3(z,0),\qquad g_3(z,t) = e^{-3a(z) t} g_3(z,0),\quad t\geq 0.
\end{align*}
The decay estimate
\begin{align}\label{eq:FPEmodediag}
|y_3(z,t)|_2^2 \leq e^{-2a(z)t}|y_3(z,0)|^2_2, \quad t\geq 0
\end{align}
follows.
\subsubsection{Case $k= 3$ and $z\in\R$ such that $a_z(z) \neq 0$:}
In this case the eigenvalue $\lambda_{2,3}(z) =1$ is defective, but this eigenvalue does not correspond to the spectral gap $\mu_3=\frac{1}{3}$.

The matrix $C_3(z)$ corresponds to two Jordan blocks. In notation from \S \ref{sec:ODE} this means: $N=2$, $l_1=1$, $l_2=2$ and $M=1$. We use Theorem \ref{th:EuclidDecay} with the modification of Remark \ref{rem:Ptime} for $n_2=2$. The (generalized) eigenvectors of $C_3^H(z)$ are given as
\begin{align*}
&v_{1,3}^{(0)} =(1, 0 ,0)^T,\\
&v_{2,3}^{(0)}(z) = (0,~ \alpha(z),~ 0)^T,&& \quad v_{2,3}^{(1)}(z) = (\textstyle\sqrt{\frac{3}{2}}\alpha(z),\, 0,\, 1)^T.
\end{align*}
With the modifications described in Remark \ref{rem:Ptime}, the matrix $\wt{P}_3(z,t)$ is constructed with three arbitrary weights: We choose them as $\beta_{1,3}^1(z) = 1$, $\beta_{2,3}^1(z)=\alpha(z)^{-2} $ and $\beta_{2,3}^2(z) = 1$, which leads to
\begin{align*}
\wt{P}_3(z,0) =
\begin{pmatrix}
1+\frac{3}{2} \alpha(z)^2  ~~~~ & 0 ~~~~& \textstyle\sqrt{\frac{3}{2}}\alpha(z)\\
0 & 1~~~~ & 0 \\
\textstyle\sqrt{\frac{3}{2}}\alpha(z) & 0~~~~ & 1 \\
\end{pmatrix}.
\end{align*}

Then, Remark \ref{rem:Ptime} (with the rescaling $t\mapsto  3a(z)t$) leads to the decay estimate
\begin{align*}
|y_3(z,t)|^2_2 \leq
\wt{\mathscr{C}}_3(z) e^{-2 a(z) t}|y_3(z,0)|^2_2,\quad t\geq 0,
\end{align*}
with the constant
\begin{align*}
\wt{\mathscr{C}}_3(z) &= \frac{\lambda_{\max}^{\wt{P}_3(z,0)}}{\lambda_{\min}^{\wt{P}_3(z,0)}}12\cdot \max\{2,1+\alpha(z)^2\}.\numberthis \label{eq:FPEuniC2}
\end{align*}

Denoting $\delta(z):=1+\frac34 \alpha(z)$, the eigenvalues of $\wt{P}_3(z,0)$ are given as
\begin{align*}
&\lambda^{\wt{P}_3(z,0)}_{1,2} = \delta(z) \pm \sqrt{\delta(z)^2-1},\qquad \lambda^{\wt{P}_3(z,0)}_{3} = 1,
\end{align*}
with $\lambda_{\max}^{\wt{P}_3(z,0)}=\lambda^{\wt{P}_3(z,0)}_{1}$ and $\lambda_{\min}^{\wt{P}_3(z,0)}=\lambda^{\wt{P}_3(z,0)}_{2}$. Since
\begin{align*}
\frac{\lambda_{\max}^{\wt{P}_3(z,0)}}{\lambda_{\min}^{\wt{P}_3(z,0)}}&=2\delta(z)^2 -1+2 \delta(z)\sqrt{\delta(z)^2-1}\\
&\leq 4 \delta(z)^2-1=3 + \frac94 \alpha(z)^4 + 6 \alpha(z)^2,
\end{align*}
the constant is uniformly bounded in $z$ by
\begin{align*}
\wt{\mathscr{C}}_3(z)\leq (6+ \frac{21}{4}\frac{\|a_z\|_\infty^4}{a_0^4})12\cdot\max\{2,1+\frac{\|a_z\|_\infty^2}{a_0^2}\}=:\mathscr{C}_3.
\end{align*}
We arrive at
\begin{align}\label{eq:FPEmode3}
|y_3(z,t)|^2_2 \leq
\mathscr{C}_3 e^{-2 a(z) t}|y_3(z,0)|^2_2
\end{align}
for $t\geq 0$.

\subsubsection{Case $k\geq 4$ and $z\in\R$:}
In this case the equations for the $k$th mode \eqref{eq:fpemodek3} does not correspond to the spectral gap $a(z)$ of the system \eqref{eq:FPE}--\eqref{eq:sensFPE}. In fact, the exponential decay rate of $|y_k(z,t)|^2_2$, $k\geq 4$ is at least $4a(z)$, which is double the rate of the slowest modes $k=1,3$. Thus, there is more freedom of choice for the matrix $P_k(z)$ for $k\geq 4$. This is important in order to get a uniform in $z$ and  $k$ estimate for $k\geq 4$. The additional difficulty compared to $k=3$ is the uniform bound for $k\to \infty$. Indeed, even Remark \ref{rem:Ptime} would not give a matrix $\wt{P}_k(z,0)$ with uniform condition number for fixed $z$ and $k\to \infty$, and therefore a decay estimate constant $\wt{\mathscr{C}}_k(z)$ that is unbounded in $k$.

The following lemma builds on the fact that the Euclidean norm, i.e.\ using $\wt{P}=I$, yields already a ``sufficient'' decay estimate as long as $|\alpha(z)|$ is small enough. An appropriate rescaling of the third coordinate via a modified norm does the trick for all $z\in\R$.

\begin{lemma}\label{lem:k4}
	For $k\geq 4$ and $z\in\R$ solutions to \eqref{eq:fpemodek3} satisfy the decay estimate
	\begin{align}\label{eq:lemk4}
	|y_k(z,t)|_2^2 \leq \mathscr{C}_{\geq4} e^{-2a(z)t}|y_k(z,0)|^2_2, \quad t\geq 0,
	\end{align}
	with the constant $\mathscr{C}_{\geq 4}:= 2 (1+\frac{\|\partial_z a\|_\infty^4}{a_0^4})$.
	\end{lemma}
The elementary but technical proof is deferred to Appendix \ref{pr:k4}.

Combining the above five cases for $k\in\N_0$ and $z\in\R$ leads to the desired uniform-in-$z$ decay estimate for arbitrary initial conditions on $L^2((f^\infty)^{-1})\times L^2((f^\infty)^{-1})$ with sharp rate:
\begin{theorem}\label{th:sensFPE}
	Let $a\in C^1(\R)$ where $a_0:=\inf_{z\in\R}a(z)>0$ and $\partial_z a\in L^\infty(\R)$. Then, there exists a constant $\mathscr{C}>0$, such that  normalized solutions $\Phi(x,z,t)=(f,g)^T$ of the system \eqref{eq:FPE}--\eqref{eq:sensFPE} with steady state $\Phi^\infty(x,z) := (f^\infty,g^\infty)^T$ satisfy
	\begin{align*}
	\sup_{z\in\R}\|\Phi(\cdot,z,t)-&\Phi^{\infty}(\cdot,z)\|^2_{L^2((f^{\infty})^{-1})}\\
	& \leq \mathscr{C}(1+t^2)e^{-2a_0t} \sup_{z\in\R}\|\Phi(\cdot,z,0)-\Phi^\infty(\cdot,z) \|^2_{L^2((f^{\infty})^{-1})}
	\end{align*}
	for $t\geq 0$ with an explicit constant $\mathscr{C}>0$ only depending on $a_0$ and $\|\partial_z a\|_\infty$, as given in \eqref{eq:kuniformconst} below.
\end{theorem}
\begin{proof}
With Parseval's identity and the collected decay estimates \eqref{eq:FPEmodek12diag}, \eqref{eq:FPEmodek12}, \eqref{eq:FPEmodediag}, \eqref{eq:FPEmode3} and \eqref{eq:lemk4}, we obtain
	\begin{align*}
	\|\Phi(\cdot,z,t)-&\Phi^\infty(\cdot,z)\|^2_{L^2((f^{\infty})^{-1})} \\
	&= \sum_{k=1}^\infty |f_k(z,t)|^2 + |g_1(z,t)|^2 + \big|g_2(z,t)+\frac{\alpha(z)}{\sqrt{2}}\big|^2 + \sum_{k=3}^\infty |g_k(z,t)|^2 \\
	&\leq\sum_{k=1}^\infty |y_k(z,t)|^2_2\\
	&\leq \mathscr{C}_{1,2}\sum_{k=1}^2 (1 +k^2a(z)^2 t^2 )e^{-2k a(z) t}|y_k(z,0)|^2_2 + \mathscr{C}_3e^{-2 a(z) t}|y_3(z,0)|^2_2\\
	&  \qquad  +  \mathscr{C}_{\geq 4}\sum_{k=4}^\infty e^{-2 a(z) t}|y_k(z,0)|^2_2\\
	&\leq (\mathscr{C}_{1,2}+\mathscr{C}_3+\mathscr{C}_{\geq 4}) (1 + a(z)^2t^2) e^{-2a(z) t }\left(\sup_{z\in\R}\,\sum_{k=1}^\infty |y_k(z,0)|_2^2\right).
	\end{align*}
Using the fact that $(1+a(z)^2 t^2)e^{-2a(z)t}\leq (1+a_0^2t^2)e^{-2a_0 t}$ for $t\geq 0$, $z\in\R$ leads to
	\begin{align*}
		\|\Phi(\cdot,z,t)-&\Phi^\infty(\cdot,z)\|^2_{L^2((f^{\infty})^{-1})} \\
	&\leq  \mathscr{C} (1 + t^2) e^{-2a_0 t} \sup_{z\in\R}\|\Phi(\cdot,z,0)- \Phi^\infty(\cdot,z) \|^2_{L^2((f^{\infty})^{-1})}
  	\end{align*}
	for each fixed $z\in\R$, $t\geq 0$ with the constant
	\begin{align*}
	&\mathscr{C}: = 2\max\{1,a_0^2\}(\mathscr{C}_{1,2}+\mathscr{C}_3+\mathscr{C}_{\geq 4})\\
	&=2\max\{1,a_0^2\}\left[12 \cdot \max\{2,1+\frac{\|a_z\|^2_\infty}{a_0^2}\}(7+ \frac{21}{4}\frac{\|a_z\|_\infty^4}{a_0^4})+2(1+\frac{\|a_z\|_\infty^4}{a_0^4})\right]<\infty.\numberthis \label{eq:kuniformconst}
	\end{align*}	\qed	
\end{proof}

\begin{remark}
Similar to \S \ref{subsec:duhamel}, Duhamel's formula would also yield a decay estimate here. The eigenfunction modes of $f(x,z,t)$ are given explicitly due to \eqref{eq:FPEfmodeeq}. This allows us to use Duhamel's formula for the evolution equation \eqref{eq:FPEsensmodeeq}, providing the eigenfunction modes of $g_k(z,t)$ in explicit form.
\end{remark}
\subsection{Uncertain diffusion coefficient}\label{subsec:FPEdiff}
As an alternative to \eqref{eq:FPE} we consider now the following Fokker--Planck equation on $\R$ with uncertainty in the diffusion term:
\begin{align}\label{eq:FPEdiff}
\partial_t u(x,z,t) =  \partial_x(d(z)u_x(x,z,t) + x u(x,z,t) ) = :L_1(z) u(x,z,t),
\end{align}
for $x,z\in\R$, $t\geq 0$ and a diffusion coefficient $d\in C^1(\R)$ satisfying $ d_0:=\inf_{z\in\R}d(z)>0$. For $v(x,z,t):=\partial_z u(x,z,t)$ the first order linear sensitivity equation is given as
\begin{align}\label{eq:FPEdiffsens}
\partial_t v(x,z,t) = L_1(z) v(x,z,t) + d_z(z) u_{xx}(x,z,t),
\end{align}
for $ x,z\in\R, t\geq 0$. A strategy similar to the one used in \S \ref{subsec:FPEdecay} can also be applied here: The rescaled Hermite functions
\begin{align*}
\hat{h}_k(x,z):= d(z)^{-\frac12} \tilde{h}_k(x d(z)^{-\frac12})
\end{align*}
with $\tilde{h}_k(x,z)$ defined in \eqref{eq:defhermite} are an orthonormal basis of $L^2((\hat{h}_0)^{-1})$ of eigenfunctions of $L_1(z)$, i.e.
\begin{align*}
L_1(z)\hat{h}_k(\cdot,z) = -k\, \hat{h}_k(\cdot,z),\quad k\in\N_0,
\end{align*}
which determines the whole ($z$-independent) spectrum
\begin{align*}
\sigma(L_1(z))=-\N_0.
\end{align*}
It follows that the unique normalized steady state of $L_1(z)$ and the corresponding steady state for \eqref{eq:FPEdiffsens} are given as
\begin{align*}
u^\infty(x,z)&=\hat{h}_0(x,z) = \frac{1}{\sqrt{2\pi d(z)}} e^{-\frac{x^2}{2d(z)}}, \\
\quad v^\infty(x,z)& = \partial_z u^\infty(x,z) = \frac{d_z(z)}{\sqrt{2}d(z)} \hat{h}_2(x,z),
\end{align*}
respectively. An eigenfunction expansion leads to the non-defective ODE systems for the eigenfunction modes $k\geq 2$:
\begin{align*}
\partial_t \begin{pmatrix}
u_{k-2}\\
v_{k}
\end{pmatrix}
= -
\underbrace{\begin{pmatrix}
k-2& ~~~~0\\
\frac{d_z(z)}{d(z)} \sqrt{(k-1)k}& ~~~~k
\end{pmatrix}}_{\udef{A_k(z):=}}
\begin{pmatrix}
u_{k-2}\\
v_{k}
\end{pmatrix},\quad z\in\R, t\geq 0,
\end{align*}
and $\partial_t v_0(z,t) = 0$, $\partial_t v_1(z,t) = -v_1(z,t)$. 
The matrix $A_k(z)$, $k\geq 2$, has the eigenvalues $\lambda_{1,k} = k-2$ and $\lambda_{2,k}= k$, and hence, is not defective. Contrary to the models previously investigated, the FPE with added uncertainty in the diffusion term \emph{does not} result in the typical defective decay behavior for $v(x,z,t)$. The decay behavior of solutions $\phi(x,z,t):=(u,v)^T$ of the system \eqref{eq:FPEdiff}--\eqref{eq:FPEdiffsens} can hence be estimated easily by
\begin{align*}
\sup_{z\in\R}\|\phi(\cdot,z,t)-\phi^\infty(\cdot,z)\|^2_{L^2((u^\infty)^{-1})} \leq \mathscr{C}e^{-t} \sup_{z\in\R} \|\phi(z,0)-\phi^\infty(\cdot,z)\|^2_{L^2((u^\infty)^{-1})},
\end{align*}
with $\phi^\infty := (u^\infty,v^\infty)^T$, a constant $\mathscr{C}>0$ and $ t\geq 0$.

To sum up, we observe that the FPE \eqref{eq:FPE} with uncertainty in the drift term gives rise to a more complicated and interesting decay behavior.

\section{Conclusion}
	In this paper we perform a sensitivity analysis for several linear PDEs with uncertainty by a Lyapunov functional method, obtaining sharp decay rates to the global equilibrium.
	
	First, a systematic derivation of Lyapunov functionals -- in the form of modified norms -- for arbitrary linear ODE systems is given. The Lyapunov functional approach has a simple geometric interpretation: In the deformed metric, the angle between any trajectory and the level curves of the $P$-norm is uniformly bounded away from zero (for $P$ constant in $t$). The novelty here is the inclusion of defective ODEs, which demand time-dependence in the norms $|\cdot |_{P(t)}$ in order to obtain sharp decay estimates of order $(1+t^M)e^{-\mu t}$. This approach is realized via a matrix $P(t)$, which is constructed from the explicit (generalized) eigenvectors of the system matrix accompanied by arbitrary weights. In the presence of an uncertainty parameter $z$, we obtain decay estimates that are uniform in $z$, which includes non-defective limits. In such cases, the matrix $P(z,t)$ has to be constructed more carefully, taking advantage of the non-uniqueness of $P(z,t)$ in the above method.

	This method is applied to three PDEs, a convection-diffusion equation, a two-velocity BGK equation, and a Fokker-Planck equation, where each of these equations feature uncertainty in the equation parameters. A linear sensitivity analysis is performed, where for the convection-diffusion equation a second order sensitivity is also included. The analysis works well with PDEs that allow for a Fourier mode decomposition, since each mode evolves according to an ODE. Hence, the decay estimates have to be uniform in the eigenmodes $k$. In the presented examples (with the exception of \S \ref{subsec:FPEdiff}) defects appear in the resulting ODEs.
	
	Sharp decay estimates which are uniform in the uncertainty parameter $z$ were obtained for these PDEs. The technical difficulty here is the possible appearance of non-defective limits due to the $z$-dependence of the ODEs, when one considers derivatives of solutions with respect to $z$. This problem is solved with a careful choice of the matrix $P(z,t)$, exploiting the fact that its construction (and in particular its weights) is not unique. 
	
	Let us point out an aspect of the sensitivity analysis of the Fokker--Planck equation in \S \ref{sec:fpe} distinct from the other investigated PDEs: The equilibrium that solutions converge to, depends itself on the uncertainty parameter.
	
	The method could be helpful even for nonlinear PDEs with uncertainties. For those problems, usually perturbative solutions, namely solutions near global equilibria, are studied, in which the exponential decay due to the linear(-ized) hypocoercivity dominates the nonlinear growth. See e.g.\ \cite{MN, Villani-book, MB, Guo-NS, Achleitner2018, Achleitner2016} for deterministic settings and \cite{LiuJin, JinZhu, ShuJin} for inclusion of uncertainty quantification. One expects that our analysis can lead to sharper decay rates than previously used energy estimates in Sobolev spaces.
 \appendix
 \section{Appendix}\label{appendix}
 \subsection{Proof of Lemma \ref{lem:PtP0est}} \label{pr:PtP0est}
 \begin{proof}
 For arbitrary $i,j\in\{1,\ldots,m\}$ and $\alpha>0$
 	\begin{align}\label{eq:vivj}
 	v^i\otimes v^j + v^j\otimes v^i \geq -\frac{1}{\alpha} Q^i - \alpha Q^j
 	\end{align}
 	holds true, as can be directly validated in matrix representation.
 	
 	We estimate $\hat{w}^m_n(t)\otimes \hat{w}^m_n(t)$ from below by using inequality \eqref{eq:vivj} for the double sum:
 	\begin{align*}
 	\hat{w}^m_n(t) &\otimes \hat{w}^m_n(t) = \sum_{k=1}^m (\xi^k(t))^2 Q^k + \sum_{\substack{i,j\in\{1,\ldots,m\}\\ i\neq j}} \xi^i(t) \xi^j(t) v^i\otimes v^j\\
 	&\geq \sum^m_{k=1} (\xi^k(t))^2 Q^k -\sum_{\substack{k,l\in\{1,\ldots,m\}\\l<k}} \left (\frac{1}{\alpha} (\xi^k(t))^2 Q^k + \alpha (\xi^l(t))^2 Q^l \right).
 	\end{align*}
 	We reorder the double sum and notice that each term depends on only one of the two indices:
 	\begin{align*}
 	&\hat{w}^m_n(t)\otimes \hat{w}^m_n(t)\\
 	&\quad \geq \sum^m_{k=1} (\xi^k(t))^2 Q^k - \sum^m_{k=2} \sum^{k-1}_{l=1}\left( \frac{1}{\alpha} (\xi^k(t))^2 Q^k + \alpha (\xi^l(t))^2 Q^l \right)\\
 	&\quad = \sum^m_{k=1} (\xi^k(t))^2 Q^k - \sum_{k=2}^m \frac{k-1}{\alpha} (\xi^k(t))^2 Q^k - \alpha \sum_{l=1}^{m-1} (m-l)(\xi^l(t))^2 Q^l\\
 	&\quad=(1-\frac{m-1}{\alpha})(\xi^m)^2 Q^m  + \sum^{m-1}_{k=1} \left(1 - \frac{k-1}{\alpha}-\alpha(m-k)\right) (\xi^k(t))^2 Q^k.
 	\end{align*}
 	For the first coefficient to be positive, we need $\alpha> m-1$. Moreover one has $\min_{k=1,\ldots,m-1} \{1 - \frac{k-1}{\alpha}-\alpha(m-k)\}= 1-\alpha(m-1)$ and therefore
 	\begin{align*}
 	\hat{w}^m_n(t) \otimes \hat{w}^m_n(t)&\geq (1-\frac{m-1}{\alpha} )(\xi^m)^2 Q^m - (\alpha(m-1)-1) \sum_{k=1}^{m-1} (\xi^k(t))^2 Q^k,
 	\end{align*}
 	which yields the desired result for $\theta:=\frac{m-1}{\alpha}\in(0,1)$.\qed
 \end{proof}
\subsection{Proof of Theorem \ref{th:EuclidDecay}}\label{pr:EuclidDecay}
If $M=1$ there are no defective eigenvalues with real part $\mu$. In this case, the result follows from \eqref{eq:Pconstantdecay} and the estimates \eqref{eq:normequiv} with the corresponding matrix $P$.

For $M>1$, we fix an arbitrary $n\in I_\mu$ and first estimate the $P_n^m(0)$-semi-norm decay for the corresponding $m\in\{2,\ldots,l_n\}$. To achieve this, we combine the decay estimate \eqref{eq:case3} and \eqref{eq:PtP0est} that gives a lower bound on the $P_n^m(t)$-semi-norm with terms only depending on $P_n^k(0)$-semi-norms ($k\in\{1,\ldots,m\}$). This yields
\begin{align*}
&(1-\theta) |x(t)|^2_{P^m_n(0)} - \left(\frac{(m-1)^2}{\theta} - 1\right) \sum^{m-1}_{k=1} \left(\frac{t^{m-k}}{(m-k)!} \right)^2 |x(t)|^2_{P^k_n(0)} \\
&\leq |x(t)|^2_{P^m_n(t)} = e^{-2\mu t} |x(0)|^2_{P^m_n(0)}.
\end{align*}
Rearranging and dividing by $(1-\theta)$ leads to
\begin{align}\label{eq:Pm0estimate1}
\begin{aligned}|x(t)|^2_{P^m_n(0)} \leq &
\underbrace{\left(\frac{(m-1)^2}{\theta} - 1\right) \frac{1}{1-\theta}}_{\udef{d_m(\theta):=}} \sum^{m-1}_{k=1} \left(\frac{t^{m-k}}{(m-k)!} \right)^2 |x(t)|^2_{P^k_n(0)}\\
&+ \frac{1}{1-\theta} e^{-2\mu t} |x(0)|^2_{P^m_n(0)}.
\end{aligned}
\end{align}
By induction we shall show that, for arbitrary but fixed $n\in I_\mu$ and all corresponding $m\in\{1,\ldots,l_n\}$, there exists a constant $c_m>0$ only depending on $m$, such that
\begin{align}\label{eq:Pinduction}
|x(t)|^2_{P_n^m(0)} \leq \frac{1}{\displaystyle\min_{k=1,\ldots,m}\beta_n^k}c_m (1+t^{2(m-1)}) e^{-2\mu t} |x(0)|^2_{P_n(0)},\quad t\geq 0,
\end{align}
where $P_n(0)=\sum_{m=1}^{l_n} \beta_n^m P_n^m(0)$ by definition \eqref{eq:Pndefective}.

For $m=1$, the matrix $P^m_n$ is not time-dependent and \eqref{eq:case3} immediately yields \eqref{eq:Pinduction} with $c_1=\frac12$.


For the inductive step, we assume the claim is true for all $k\in\{1,\ldots,m\}$ with some $m\geq 1$ and constants $c_k>0$ monotonically increasing in $k$ and start from \eqref{eq:Pm0estimate1}, written for $m+1$:
\begin{align}
|x(t)|&^2_{P^{m+1}_n(0)} \leq d_{m+1}(\theta) \sum^m_{k=1} \left(\frac{t^{m+1-k}}{(m+1-k)!} \right)^2 |x(t)|^2_{P^k_n(0)}\nonumber \\
&\qquad \qquad + \frac{1}{1-\theta} e^{-2\mu t} |x(0)|^2_{P_n^{m+1}(0)} \nonumber \\
&\leq d_{m+1}(\theta) \sum^m_{k=1}  \frac{t^{2(m+1-k)}}{[(m+1-k)!]^2} \frac{1}{\displaystyle\min_{j\in\{1,\ldots,k\}}\beta_n^j}c_k(1+t^{2(k-1)}) e^{-2\mu t} |x(0)|^2_{P_n(0)}\label{eq:PinductionCompute}\\
&\qquad \qquad +\frac{1}{\beta_n^{m+1}}\frac{1}{1-\theta}e^{-2\mu t} |x(0)|^2_{P_n(0)}\nonumber
\end{align}
where \eqref{eq:Pinduction} with $k\in\{1,\ldots,m\}$ was used in the second estimate.

In order to combine both of the terms in \eqref{eq:PinductionCompute} into one summation, we compute $\inf_{\theta\in(0,1)}\max\{ d_{m+1}(\theta), \frac{1}{1-\theta}\}$. For $m=1$, one has $\inf_{\theta\in(0,1)}\max\{ d_{2}(\theta), \frac{1}{1-\theta}\}=2$. For $m>1$, the coefficient $d_{m+1}(\theta)$ has its minimum at $\theta_{\min}=m^2-m\sqrt{m^2-1}$ with value $
d_{m+1}(\theta_{\min})=2m^2+2m\sqrt{m^2-1}-1$. As $d_{m+1}(\theta_{\min})\geq \frac{1}{1-\theta_{\min}}$, one gets
$
\inf_{\theta\in(0,1)}\max\{ d_{m+1}(\theta), \frac{1}{1-\theta}\}=d_{m+1}(\theta_{\min}).$

In total
\begin{align}\label{eq:dmEstimate}
\inf_{\theta\in(0,1)}\max\{ d_{m+1}(\theta), \frac{1}{1-\theta}\}\leq 4m^2-1,\qquad m\geq 1.
\end{align}
Applying this estimate to \eqref{eq:PinductionCompute} and, additionally, using  the  upper bound\linebreak ${\max_{k=1,\ldots,m} t^{2(m+1-k)} + t^{2m}} \leq 2(1+t^{2m})$ for all $t\geq 0$, one gets
\begin{align*}
|x(t)|^2_{P^{m+1}_n(0)} &\leq \frac{1}{\displaystyle\min_{k\in\{1,\ldots,m+1\}}\beta_n^k}\\
&\times \underbrace{2(4m^2-1)c_m\sum_{k=1}^{m+1}\frac{1}{[(m+1-k)!]^2}}_{\udef{c_{m+1}:=}} (1+t^{2m})\,e^{-2\mu t}|x(0)|^2_{P_n(0)},
\end{align*}
which concludes the induction and hence the proof of \eqref{eq:Pinduction} for $m\in\{1,\ldots,l_n\}$.

The constant $c_{m+1}$ for $m\in\{1,\ldots,l_n-1\}$ is given as
\begin{align*}
c_{m+1}&= 2^{m-1} \left(\prod_{j=1}^{m} 4j^2-1\right) \left(\prod_{j=2}^{m+1} \sum_{k=1}^j \frac{1}{[(j-k)!]^2}\right).
\end{align*}
By definition \eqref{eq:defP} the matrix $P(0)$ is given as
\begin{align*}
P(0)=\underbrace{\sum_{n\not\in I_\mu} \beta_n P_n}_{\udef{P_{I^c_\mu}:=}}  +\underbrace{\sum_{n\in I_\mu}\sum_{m=1}^{l_n} \beta_n^m P_n^m(0)}_{\udef{P_{I_\mu}:=}}.
\end{align*}
The first term, $P_{I_\mu^c}$, covers the Cases 1--2. Applying Gronwall's lemma directly to the inequalities \eqref{eq:case1} and \eqref{eq:Pndecay} yields
\begin{align*}
|x(t)|_{P_{I^c_\mu}}^2 \leq  e^{-2 \mu t} |x(0)|_{P_{I^c_\mu}}^2.
\end{align*}
Now, we take a closer look at the decay behavior of solutions with respect to the $P_{I_\mu}$-semi-norm that corresponds to Case 3.
\begin{align}
|x(t)|_{P_{I_\mu}}^2&=\left( \sum_{n\in I_\mu} \beta_n^1 |x(t)|^2_{P_n^1(0)} + \sum_{n\in I_\mu} \sum_{m=2}^{l_n}\beta^m_n |x(t)|^2_{P^m_n(0)}\right). \label{eq:estimatePImu}
\end{align}
The first term includes only semi-norms that are time-independent, i.e. $P_n^1(t)\equiv P_n^1(0)$, and using \eqref{eq:case3} directly gives the decay behavior. For the second term in \eqref{eq:estimatePImu}, we apply \eqref{eq:Pinduction} and get
\begin{align*}
|x(t)|_{P_{I_\mu}}^2&\leq   \sum_{n\in I_\mu} \beta_n^1 e^{-2\mu t}|x(0)|^2_{P_n^1(0)}  \\
&\qquad +\sum_{n\in I_\mu} \sum_{m=2}^{l_n}\frac{\beta^m_n}{\displaystyle\min_{k\in\{1,\ldots,m\}}\beta_n^k} c_m (1+t^{2(m-1)})e^{-2\mu t} |x(0)|^2_{P_n(0)}\\
&\leq  e^{-2\mu t} \sum_{n\in I_\mu}  |x(0)|^2_{P_n(0)} \\
&\qquad +\displaystyle\max_{n\in I_\mu}
\Big[\sum_{m=2}^{l_n}\frac{\beta_n^m}{\displaystyle\min_{k\in\{1,\ldots,m\}}\beta_n^k} c_m\Big] 2(1+t^{2(M-1)})e^{-2\mu t}\sum_{n\in I_\mu} |x(0)|^2_{P_n(0)} \\
&\leq 2c_M\displaystyle\max_{n\in I_\mu}
\Big[\sum_{m=1}^{l_n}\frac{\beta_n^m}{\displaystyle\min_{k\in\{1,\ldots,m\}}\beta_n^k} \Big] (1+t^{2(M-1)})e^{-2\mu t} |x(0)|^2_{P_{I_\mu}}.
\end{align*}
Now, using \eqref{eq:normequiv} for $P(0)$, the decay behavior in the Euclidean norm follows as
\begin{align*}
|x(t)|^2_2&\leq (\lambda_{\min}^{P(0)})^{-1}|x(t)|^2_{P(0)}\\
&=(\lambda_{\min}^{P(0)})^{-1}\left( |x(t)|_{P_{I^c_\mu}}^2+ |x(t)|_{P_{I_\mu}}^2 \right)\\
&\leq  (\lambda_{\min}^{P(0)})^{-1} \Big(e^{-2\mu t} |x(0)|_{P_{I_\mu^c}}^2 \\
&\qquad +
2c_M\,\displaystyle\max_{n\in I_\mu}
\Big[\sum_{m=1}^{l_n}\frac{\beta_n^m}{\displaystyle\min_{k\in\{1,\ldots,m\}}\beta_n^k} \Big] (1+t^{2(M-1)})e^{-2\mu t} |x(0)|^2_{P_{I_\mu}}\Big)\\
&\leq 2(\lambda_{\min}^{P(0)})^{-1} \lambda_{\max}^{P(0)}c_M\,
\displaystyle\max_{n\in I_\mu}
\Big[\sum_{m=1}^{l_n}\frac{\beta_n^m}{\displaystyle\min_{k\in\{1,\ldots,m\}}\beta_n^k} \Big] (1+t^{2(M-1)})e^{-2\mu t} |x(0)|^2_2,
\end{align*}
where the constant $\mathscr{C}:= 2(\lambda_{\min}^{P(0)})^{-1}\lambda_{\max}^{P(0)}c_M\displaystyle\max_{n\in I_\mu}
\Big[\sum_{m=1}^{l_n}\frac{\beta_n^m}{\min_{k\in\{1,\ldots,m\}}\beta_n^k} \Big]$ depends only on the matrix $P(0)$.\qed
\subsection{Proof of Lemma \ref{lem:CDsecond}} \label{pr:CDsecond}
		\begin{proof}
			Since $w_{1,k}^2,w_{1,k}^3$ satisfy \eqref{eq:connectionCw} with $n=1$ and $m=2,3$, their linear combination $\wt{w}_{k}^3$ satisfies
			\begin{align*}
			\frac{d}{dt}\wt{w}_k^3(z,t) = (D_k^H(z)-\overline{\lambda}_k(z))\wt{w}_k^3(z,t).
			\end{align*}
			The computation leading to \eqref{eq:case3} also applies here (with rescaled time $\tau_k= k^2t$) and results in
			\begin{align}\label{eq:tildePtdecay}
			|y_k(z,t)|^2_{\wt{P}^3_k(z,k^2 t)} = e^{-2 k^2 b(z) t} |y_k(z,0)|_{\wt{P}_k^3(z,0)}^2.
			\end{align}

			Now, our goal is an estimate in the $\wt{P}_k^3(z,0)$-semi-norm with help of Lemma \ref{lem:PtP0est}. By definition
			\begin{align*}
			\wt{w}_k^3(z,t) &= \xi_k^1(z,t) w^1_{1,k}(z,0)+  \xi_k^2(z,t) w^2_{1,k}(z,0)
			+ \xi_k^3 \wt{w}^3_k(z,0),
			\end{align*}
			with the polynomials
			\begin{align*}
			\xi_{k}^1(z,t)=\frac{t^2}{2} + \frac{\partial^2_z \overline{\lambda}_k(z)}{2(\partial_z \overline{\lambda}_k(z))^2}t, \quad \xi^2_k(z,t) = t \quad\text{and}\quad \xi^3_k(z,t) = 1.
			\end{align*}
			Lemma \ref{lem:PtP0est} yields
			\begin{align}\label{eq:xiupper}
			\begin{aligned}|x|^2_{\wt{P}_{k}^3(z,t)} &\geq (1-\theta) |x|^2_{\wt{P}_k^3(z,0)} \\
			&\qquad - \left(\frac{4}{\theta} - 1\right) \left[ |\xi^1_{k}(z,t)|^2 |x|^2_{P^1_{1,k}(z,0)} + |\xi^2_k(z,t)|^2 |x|^2_{P^2_{1,k}(z,0)} \right]
			\end{aligned}
			\end{align}
			for any (fixed) $x\in\C^3$, $\theta\in(0,1)$, $t\geq 0$. Replace $x$ by a solution $y_k(z,t)$ to \eqref{eq:secorderODE} (rescaling $\xi_k(z,t)$ to $\xi_k(z,k^2t)$ to account for the prefactor $k^2$ in the ODE). Then (in analogy to the estimate \eqref{eq:Pm0estimate1}), using \eqref{eq:tildePtdecay} leads to
			\begin{align}\label{eq:contPtilde}
			\begin{aligned}|y_k(z,t)|^2_{\tilde{P}^3_k(z,0)}& \leq d_3(\theta)\left[ |\xi^1_{k}(z,k^2t)|^2 |y_k(z,t)|^2_{P^1_{1,k}(z,0)} +  |\xi^2_{k}(z,k^2t)|^2 |y_k(z,t)|^2_{P^2_{1,k}(z,0)} \right] \\
			&\qquad + \frac{1}{1-\theta} e^{-2k^2 b(z)t} |y_k(z,0)|^2_{\tilde{P}^3_k(z,0)},
			\end{aligned}
			\end{align}
			with $d_3(\theta):=\frac{4-\theta}{\theta(1-\theta)}$ as in \eqref{eq:Pm0estimate1}. Minimizing in $\theta$ with estimate \eqref{eq:dmEstimate} yields
			\begin{align*}
			\inf_{\theta\in(0,1)}\max\left\{ d_3(\theta), \frac{1}{1-\theta}\right\} \leq 15.
			\end{align*}
			Next we use the estimates \eqref{eq:Pk1}, \eqref{eq:Pk2} and the fact that $\wt{P}^3_k(z,0)\leq \frac{1}{4|\partial_z \lambda_k(z)|^4} I$ to proceed with \eqref{eq:contPtilde}:
			\begin{align*}
			|y_k(z,t) |^2_{\wt{P}_k^3(z,0)} &\leq 15 \left[ \Big|\frac{k^4t^2}{2} + k^2t\frac{\partial^2_z \overline{\lambda}_k(z)}{2(\partial_z \overline{\lambda}_k(z))^2}\Big|^2  +  \frac{6 k^4t^2}{\min\{1,|\partial_z \lambda_k(z)|^2\}} (1+k^4t^2 )\right.\\
			&\qquad \qquad \qquad \qquad \qquad \qquad\quad  \left.  +\frac{1}{4|\partial_z \lambda_k(z)|^4}\right] e^{-2k^2 b(z) t} |y_k(z,0)|^2_2 \nonumber\\
			&\leq 15\frac{1+|\partial_z^2 \lambda_k(z)|^2}{\min\{1,|\partial_z \lambda_k(z)|^4\}} \left[  \frac{k^8t^4}{4} + \frac{k^4t^2}{4} + \frac{k^6 t^3}{2}\right.\\
			& \qquad \qquad \qquad \qquad \quad\quad \left .+6 (k^4t^2+k^8t^4)+\frac14 \right] e^{-2k^2 b(z) t} |y_k(z,0)|^2_2\\
			&\leq 146.25 \frac{1+|\partial_z^2 \lambda_k(z)|^2}{\min\{1,|\partial_z \lambda_k(z)|^4\}} (1+k^8t^4)e^{-2k^2 b(z)t} |y(z,0)|^2_2.
			\end{align*}\qed
	\end{proof}

\subsection{Proof of Lemma \ref{lem:k4}}\label{pr:k4}
\begin{proof}
	To show that the matrix
	\begin{align}
	\wt{P}(z):=
	\begin{pmatrix}
	~1 ~~~~& 0 ~~~~&0\\
	~0~~~~&1~~~~&0\\
	~0~~~~&0~~~~&\frac12 \min\{1,\frac{1}{\alpha(z)^4}\}
	\end{pmatrix}
	\end{align}
	satisfies
	\begin{align}\label{eq:cppck4}
	C_k^H(z) \wt{P}(z) + \wt{P}(z) C_k(z) \geq \frac{1}{2} \wt{P}(z), \quad z\in\R, k\geq 4,
	\end{align}
	we show that
	\begin{align*}
	A_k(z)&:=C_k^H(z) \wt{P}(z) +\wt{P}(z) C_k(z) - \frac{1}{2} \wt{P}(z)\\
	&=
	\begin{pmatrix}
	\frac{3}{2}-\frac{4}{k}&~~~ 0 ~~~& \frac12 \gamma(k)\min\{\alpha(z),\frac{1}{\alpha(z)^3}\}\\
	0 & \frac{3}{2} & \frac12 \min\{\alpha(z),\frac{1}{\alpha(z)^3}\}\\
	\frac12 \gamma(k)\min\{\alpha(z),\frac{1}{\alpha(z)^3}\} &~~~ \frac12 \min\{\alpha(z),\frac{1}{\alpha(z)^3}\}~~~ & \frac34 \min\{1,\frac{1}{\alpha(z)^4}\}
	\end{pmatrix}
	\end{align*}
	is positive definite.
	
	As $k\geq 4$, the first two leading minors are positive. The third minor is positive, if
	\begin{align*}
	\det A_k(z)& = \frac{9}{8}\left(\frac32-\frac{4}{k}\right)\min\{1 , \frac{1}{\alpha(z)^4}\}-\frac{3}{8} \gamma(k)^2 \min\{\alpha(z)^2, \frac{1}{\alpha(z)^6}\} \\
	&\qquad  - \frac{1}{4}\left(\frac32-\frac{4}{k}\right) \min\{\alpha(z)^2 ,\frac{1}{\alpha(z)^6}\}>0
	\end{align*}
	for all $z\in\R$, $k\geq 4$. For all $k\geq 4$, we distinguish the following two cases:
	
	\emph{For $z\in\R$ such that $|\alpha(z)|\geq 1$} the condition $\det A_k(z)>0$ is equivalent to
	\begin{align*}
	f(k,\alpha(z)^2,\gamma(k)):= \left(\frac32-\frac{4}{k}\right)\Big(\frac{9}{4}-\frac{1}{2\alpha(z)^2}\Big)-  \frac{3}{4\alpha(z)^2}\gamma(k)^2>0.
	\end{align*}
	The function $[4,\infty)\times [1,\infty) \times [\sqrt{\frac{2}{3}},1) \ni (k,\alpha^2,\gamma)\mapsto f(k,\alpha^2,\gamma)$,  is monotonously increasing in $k$ and $\alpha^2$ but monotonously decreasing in $\gamma$, hence
	\begin{align*}
	f(k,\alpha(z)^2,\gamma(k)) \geq f(4,1,1) = \frac18 >0.
	\end{align*}

	\emph{For $z\in\R$ such that $|\alpha(z)|\leq 1$} the condition $\det A_k(z)>0$ is equivalent to
	\begin{align*}
	g(k,\alpha(z)^2,\gamma(k)):= \left(\frac32-\frac{4}{k}\right)\Big(\frac{9}{4}-\frac12 \alpha(z)^2\Big) -  \frac34\gamma(k)^2 \alpha(z)^2 >0.
	\end{align*}
	The function $[4,\infty)\times [0,1] \times  [\sqrt{\frac{2}{3}},1) \ni (k,\alpha^2,\gamma)\mapsto g(k,\alpha^2,\gamma)$ is monotonously increasing in $k$ and monotonously decreasing in $\alpha^2$ and $\gamma$, hence
	\begin{align*}
	g(k,\alpha(z)^2,\gamma(k))\geq g(4,1,1) &= \frac{1}{8}>0.
	\end{align*}
	
	This proves the matrix inequality \eqref{eq:cppck4}. With a similar calculation as \eqref{eq:case1} (and the rescaling $t\mapsto k\alpha(z)t$), this implies
	\begin{align*}
	|y_k(z,t)|_{\wt{P}(z)}^2 \leq e^{-\frac12 ka(z)t}|y_k(z,0)|^2_{\wt{P}(z)}, \quad t\geq 0, k\geq 4.
	\end{align*}
	With $\frac12 \min\{1,\frac{1}{\alpha(z)^4}\}I\leq \wt{P}(z)\leq I$, $z\in\R$, we obtain
	\begin{align*}
	|y_k(z,t)|^2_2 &\leq 2 \max\{1,\alpha(z)^4\} e^{-2a(z)t}|y_k(z,0)|^2_2, \quad t\geq 0, k\geq 4,
	\end{align*}
	from which the desired result follows.\qed
\end{proof}

\bibliographystyle{spmpsci} 


\end{document}